\documentclass[oneside,english]{amsart}
\usepackage[T1]{fontenc}
\usepackage[latin9]{inputenc}
\usepackage{geometry}
\usepackage{amsthm}
\usepackage{amssymb}
\usepackage{amsfonts}
\usepackage{amsmath}
\usepackage[all]{xy}

\makeatletter
\numberwithin{equation}{section}
\numberwithin{figure}{section}
\theoremstyle{plain}
\newtheorem{thm}{\protect\theoremname}
  \theoremstyle{plain}
  \newtheorem{prop}[thm]{\protect\propositionname}
  \theoremstyle{definition}
  \newtheorem{defn}[thm]{\protect\definitionname}
  \theoremstyle{remark}
  \newtheorem{rem}[thm]{\protect\remarkname}
  \theoremstyle{plain}
  \newtheorem{lem}[thm]{\protect\lemmaname}
  \theoremstyle{plain}
  \newtheorem{cor}[thm]{\protect\corollaryname}

\makeatother

\usepackage{babel}
  \providecommand{\corollaryname}{Corollary}
  \providecommand{\definitionname}{Definition}
  \providecommand{\lemmaname}{Lemma}
  \providecommand{\propositionname}{Proposition}
  \providecommand{\remarkname}{Remark}
\providecommand{\theoremname}{Theorem}

\newcommand{\C}{\mathbb{C}}
\newcommand{\R}{\mathbb{R}}
\newcommand{\qq}{\mathbb{Q}}

\newcommand{\ZZ}{\mathbb{Z}}

\newcommand{\Spec}{\mathrm{Spec}}

\newcommand{\GL}{\mathrm{GL}}

\newcommand{\Hom}{\mathrm{Hom}}

\newcommand{\mc}{\mathcal}
\newcommand{\mf}{\mathfrak}
\newcommand{\mb}{\mathbb}

\newcommand{\inj}{\hookrightarrow}
\newcommand{\sub}{\subseteq}
\newcommand{\Zp}{\mathbb{Z}_{p}}
\newcommand{\Qp}{\mathbb{Q}_{p}}
\newcommand{\Fp}{\mathbb{F}_{p}}
\newcommand{\dg}{\dagger}
\newcommand{\ra}{\rightarrow}
\newcommand{\ol}{\overline}
\newcommand{\wh}{\widehat}

\begin{document}

\title{A trace formula approach to control theorems for overconvergent automorphic
forms}

\author{Christian Johansson}
\begin{abstract}
We present an approach to proving control theorems for overconvergent
automorphic forms on Harris-Taylor unitary Shimura varieties based
on a comparison between the rigid cohomology of the multiplicative
ordinary locus and the rigid cohomology of the overlying Igusa tower,
the latter which may be computed using the Harris-Taylor version of
the Langlands-Kottwitz method. We also prove a higher level version,
generalizing work of Coleman.
\end{abstract}
\maketitle

\section{Introduction}

In \cite{Col}, Coleman proved that if $f$ is an overconvergent modular
form of weight $k$ and tame level $\Gamma_{1}(N)$ which is an eigenform
for $U_{p}$ with slope (i.e. the $p$-adic valuation of the eigenvalue)
less than $k-1$, then $f$ is in fact a (classical) modular form
of weight $k$ for the congruence subgroup $\Gamma_{1}(N)\cap\Gamma_{0}(p)$
using an analysis of the cohomology of the ordinary locus of the modular
curve. This theorem was later reproved by Kassaei (\cite{Kas}) using
a very different geometric method inspired by work of Buzzard and
Taylor on the Artin conjecture (\cite{BT}). Since then this method
has been generalised more general Shimura varieties, but no progress
had been made on the cohomological method. In \cite{Joh} we presented
a cohomological argument essentially in the context of Hilbert modular
varieties, combining the main conceptual ideas of Coleman with some
slope calculations in rigid cohomology. Independently, Tian and Xiao
(\cite{TX}) gave a similar cohomological argument for Hilbert modular
varieties, utilising geometric results on the Ekedahl-Oort stratification
instead of slopes in rigid cohomology. In this paper we wish to present
an alternative to these two arguments using the cohomology of Igusa
varieties, illustrated in the context of some unitary Shimura varieties
studied by Harris and Taylor in their proof of the Local Langlands
correspondence. Our method extends easily to give a higher level version, giving a higher rank analogue of \cite[Theorem 1.1]{Col2}.

Let us briefly sketch the argument of \cite{Col}, in a language closer
to that of \cite{Joh,TX} and this paper. Let $Y$ be the
(non-compact) modular curve of level $\Gamma_{1}(N)$ over $\Zp$
($p\nmid N$) and let $\mc{E}_{k}$ be the local system $Sym^{k-2}(H^{1}(E/Y))$
in any appropriate cohomology theory where $E/Y$ is the universal
elliptic curve. Coleman proved that the rigid cohomology group $H_{rig}^{1}(Y_{\Fp}^{ord},\mc{E}_{k})$
is isomorphic to a quotient $M_{k}^{\dg}/\theta^{k-1}M_{2-k}^{\dg}$
where $M_{k^{\prime}}^{\dg}$ is the sheaf of weight $k^{\prime}$
overconvergent modular forms, $Y_{\Fp}^{ord}$ is the ordinary
locus, and further showed that if $f\in M_{k}^{\dg}$ is a generalized $U_{p}$-eigenvector
of slope less than $k-1$, then $f\notin\theta^{k-1}M_{2-k}^{\dg}$.
By comparing $H_{rig}^{1}(Y_{\Fp}^{ord},\mc{E}_{k})$
to the parabolic cohomology $H_{par}^{1}(Y_{\Qp},\mc{E}_{k})$
and some dimension counting Coleman computed $H_{rig}^{1}(Y_{\Fp}^{ord},\mc{E}_{k})$
in terms of classical modular forms and proved that such $f$ as in
the previous sentence are classical. A key point here is that the
small slope forms whose cohomology class comes from $H_{par}^{1}(Y_{\Qp},\mc{E}_{k})$
makes up ``most'' of $H_{rig}^{1}(Y_{\Fp}^{ord},\mc{E}_{k})$
and the contribution from the classical forms not of this form may
be quantified and shown to exhaust the rest of $H_{rig}^{1}(Y_{\Fp}^{ord},\mc{E}_{k})$.

The isomorphism $H_{rig}^{1}(Y_{\Fp}^{ord},\mc{E}_{k})\cong M_{k}^{\dg}/\theta^{k-1}M_{2-k}^{\dg}$
may be generalised to very general compact Shimura varieties using Faltings's dual BGG complex; the argument presented
in \cite{Joh} readily generalises as long as the basic ingredients
are present (the non-compact case is more involved and requires restricting to cusp forms, see \cite{AIP}, \cite{HLTT} and \cite{Str}) . It is also insensitive to whether one works with the
Shimura variety with full level or with Iwahori level at $p$ (the
former is the case covered in \cite{Joh}, the latter is covered in
this paper). In fact, in the setting of this paper the argument  works well even with full $p^{n}$-structure for any $n$. Similarly the observation that if $f$ has small slope
then $f\notin\theta^{k-1}M_{2-k}^{\dg}$ also generalises, using the observation in \cite{Joh} that this is
a consequence of the fact that the central character remains constant
throughout the dual BGG complex. These results, though proven in slightly
different ways in the cases at hand, are common to \cite{Joh}, \cite{TX}
and this paper and form the first half of the argument. It remains
to carry out the analogue of computing $H_{rig}^{1}(Y_{\Fp}^{ord},\mc{E}_{k})$.
In \cite{Joh}, we only computed a part of this group, which sufficed
to deduce a control theorem; in \cite{TX} the whole group was computed
using the geometric results mentioned above.

In this paper we propose that the Euler characteristic $\sum_{i}(-1)^{i+1}H_{rig}^{i}(Y_{\Fp}^{ord},\mc{E}_{k})$
may be computed in terms of classical automorphic representations
using the Igusa tower above $Y_{\Fp}^{ord}$ and a comparison
of the Lefschetz trace formula in rigid cohomology and the Arthur-Selberg
trace formula. We will be working in the context of the book \cite{HT}
in the special case when the totally real field is $\qq$;
let us now briefly recall this setting (for precise definitions see
\S 2). We let $G$ be a unitary group over $\qq$ coming from
a division algebra $B$ over an imaginary quadratic field $F$ satisfying
a list of conditions, the most important being that $G(\R)\cong GU(n-1,1)$
and that for our fixed prime $p$, $G(\Qp)\cong{\rm GL}_{n}(\Qp)\times\qq_{p}^{\times}$.
Associated with these groups are Shimura varieties $X_{U}$ ($U$
compact open subgroup of $G(\mb{A}^{\infty})$) which may be defined over $F$; when $U$ is of a certain type at $p$ the $X_{U}$ have proper
integral models that were extensively studied in \cite{HT} and \cite{TY}.
For the purposes of this introduction let us focus on the case when the level at $p$ is
Iwahori; let us call this open compact $Iw=U^{p}Iw_{p}$. The integral
models of $X_{Iw}$ have strictly semistable reduction; the special
fibre $Y_{Iw}$ has an ordinary-multiplicative locus $Y_{Iw,1}^{0}$
we will be interested in together with its tubular neighbourhood $X_{Iw}^{ord}$
inside the analytification of $X_{Iw}$. Our Shimura varieties have
overconvergent $F$-isocrystals $V^{\dg}(\xi)$ associated with
algebraic representations $\xi$ of $G$ as well as coherent sheaves
$W(\mu)$ associated with algebraic representations $\mu$ of the
parabolic subgroup $Q$ of $G_{\C}$ attached to the Shimura
datum; overconvergent automorphic forms are defined as overconvergent
sections of the $W(\mu)$ on $X_{Iw}^{ord}$.

After carrying out the steps sketched above, one arrives with an
overconvergent automorphic form of small slope appearing in $H_{rig}^{d}(Y_{Iw,1}^{0},V^{\dg}(\xi))$,
where $d=\dim\, X_{Iw}$. In fact, we make a slight refinement of
the arguments to make sure that it appears $\sum_{i}(-1)^{d+i}H_{rig}^{i}(Y_{Iw,1}^{0},V^{\dg}(\xi))$.
On top of $Y_{Iw,1}^{0}$ (as well as the other strata of $Y_{Iw}$)
Harris-Taylor and Taylor-Yoshida constructed a tower of Igusa varieties
"of the first kind"; a tower of finite Galois covers $(Ig_{U^{p},m})_{m\geq1}\ra Ig_{Iw}\cong Y_{Iw,1}^{0}$,
and computed the corresponding Euler characteristics $\sum_{i}(-1)^{d+i}H_{et}^{i}(Ig_{Iw},V(\xi))$
in terms of $\sum_{i}(-1)^{d+i}H_{dR}^{i}(X_{Iw},V(\xi))$. This is
one of the main technical results of \cite{HT} (Theorem V.5.4), known
as the ``second basic identity''. We use a comparison of Lefschetz
trace formulas in \'etale and rigid cohomology to transfer this result
to rigid cohomology. The comparison between $\sum_{i}(-1)^{d+i}H_{et}^{i}(Ig_{Iw},V(\xi))$
and $\sum_{i}(-1)^{d+i}H_{dR}^{i}(X_{Iw},V(\xi))$ is straightforward
for Hecke operators at primes away from $p$; the remaining work is
to work out what happens at $p$. In the end, one arrives at the following
theorem, which is the main result of this paper (Theorem
\ref{thm: main}) in the case of Iwahori level:
\begin{thm}
Assume that $k_{1}>...>k_{n-1}>k_{n}+n$. Let $f$ be
an overconvergent generalized Hecke eigenform of weight $(k_{1},...,k_{n})$ and $U_{p}$-slope less
than $k_{n-1}-k_{n}+1-n$. Then $f$ is classical.
\end{thm}

Here $U_{p}$ is a Hecke operator at $p$ generalising
the classical $U_{p}$ for $GL_{2}$ (see \S \ref{sub:Hecke-algebras}
for the definition) and we use tuples of integers $(k_{1}\geq...\geq k_{n-1},k_{n},w)$
as dominant weights to classify the algebraic representations $\mu$
of the Levi of $P$ (in the theorem we made a specific choice of $w$ for simplicity, correspond to the standard arithmetic normalization of the action of Hecke operators). The same strategy with very little extra effort also gives a higher level analogue (also recorded in Theorem \ref{thm: main}).

The main advantages of the strategy employed
in this paper compared that of \cite{Joh} is that one gets stronger
results, one can work at the Iwahori level (or higher level) as opposed to full level
at $p$, and although one needs an integral model, one needs to know
very little about it apart from the properness of the whole model
and smoothness and moduli interpretation of the ordinary locus (note
that we do not use the moduli interpretation of the rest of the special
fibre). The second point is important, because the Iwahori level is
the right one for constructing the eigenvariety (see \cite{AIP}). In
\cite{Joh} this was not a problem, as in that setting the canonical
subgroup allows one to pass from full level to Iwahori level, but
in general (e.g. in our setting here for $n\geq3$) the canonical
subgroup only takes one to some parahoric (non-Iwahori) level. Thus
any control theorem proved on the full level Shimura variety needs
to be complemented with descent results before it can be applied to
the eigenvariety. The main disadvantage is of course that it relies
heavily on the technical results on the cohomology of the Igusa tower
proved in \cite{HT}. The construction of the Igusa tower has been
generalised to the unramified PEL (type $A$ and $C$) setting by
Mantovan (\cite{Man}) and their points have been counted and the
resulting formula stabilised by Shin (\cite{Shi1,Shi2}). It remains
however to compare the resulting formula to the Arthur-Selberg trace
formula in full generality. We should remark though that many compact
cases with ``no endoscopy'' were done in \cite{Shi4} and some compact
cases with endoscopy were done in \cite{Shi3}; we expect that
our methods generalise directly to these cases.

Control theorems for the Shimura varieties considered in this paper have been proved previously in the Iwahori case using the analytic continuation method by Pilloni and Stroh (\cite{PS2}) and later vastly generalized and extended to include higher level cases by Bijakowski in a series of papers (\cite{Bij,Bij2,Bij3}). We remark that the only reason for restricting our attention in this paper to the special case when $F$ is an imaginary quadratic field (in the context of the Shimura varieties considered in \cite{HT}, as opposed to $F$ being a CM field containing an imaginary quadratic field) was the need for the $\mu$-ordinary locus to be affine. This is known when the $\mu$-ordinary locus is actually the ordinary locus (i.e. when the ordinary locus is non-empty, this is a consequence of the existence of the Hasse invariant) and the assumption that $F$ is imaginary quadratic (together with the other assumptions made in \cite{HT}) forces the ordinary locus to be non-empty. The ordinary locus can be non-empty in other cases as well, but we felt that the extra generality would only add to the complexity of the notation while the proofs would still be all the same. After the first version of this paper this paper was written we learned that Goldring and Nicole (\cite{GN}) has proved the affineness of the $\mu$-ordinary locus for general unitary Shimura varieties at a prime of good reduction. In particular, this means that one can run the arguments of this paper in the case when $F$ is CM with $p$ unramified to obtain control theorems for Shimura varieties with vanishing ordinary locus, as long as one can extend Hida's calculations for integrality of Hecke operators at $p$. These calculations were done in \cite{Bij3}, which proves control theorems in at Iwahori level in many cases with empty ordinary locus. With this, the arguments of this paper generalise to give higher level control theorems as well.

Let us now outline the contents of the papers. Section \ref{sec: definitions}
is devoted to setting up the definitions of the groups, Shimura varieties
and the Igusa varieties that we will be working with. In \S \ref{sec:Hecke-Actions}
we discuss the relation between our Shimura varieties and Igusa varieties
as towers with the action of adelic groups and define the automorphic
sheaves and the Hecke algebras that we will be working with, and record
a few relations. Finally, the computation of the cohomology is carried
out in \S \ref{sec:Computation-of-Cohomology}, first from the point
of view of overconvergent automorphic forms and then in terms of classical
automorphic representations. The comparison between the two viewpoints
gives the main theorem.

\subsection*{Acknowledgements}

It is a pleasure to thank my PhD adviser Kevin Buzzard for suggesting
I try to generalise Coleman's method for my thesis and for his help
and advice. I wish to thank Teruyoshi Yoshida for enlightening conversations
and in particular for answering a question that led to the discovery
of this method, as well as for sending me an updated version of his
Harvard PhD thesis. I also wish to thank
James Newton, Vincent Pilloni and Shu Sasaki for useful conversations, as well as an anonymous referee for comments and corrections. During my doctoral studies I was supported by the EPSRC and it is
a pleasure to thank them, as well as thanking the Fields Institute
for their support and hospitality during a stay in early 2012 when
this work was begun. The author was also supported by EPSRC Grant EP/J009458/1 and NSF Grant DMS-1128155 during various stages of the preparation and revision of this paper.

\section{Groups and Shimura Varieties\label{sec: definitions}}

\subsection{The groups}

The main reference for this section is \cite[\S I.7]{HT}. We will,
to the best of our ability, follow the notation and terminology set
up there. Let $n\geq2$ be a fixed integer, $p$ a fixed rational
prime and $F$ an imaginary quadratic field in which $p$ splits as $p=uu^{c}$, where $-^{c}$ denotes the nontrivial automorphism of $F$ (complex conjugation). We will let $B$ denote a division algebra with centre $F$ such that
\begin{itemize}
\item ${\rm dim}_{F}B=n^{2}$;
\item $B^{op}\cong B\otimes_{F,c}F$;
\item $B$ is split at $u$;
\item at any non-split place $x$ of $F$, $B_{x}$ is
split;
\item at any split place $x$ of $F$ either $B_{x}$ is split or $B_{x}$
is a division algebra;
\item if $n$ is even, then $1+n/2$ is congruent modulo $2$ to the number
of places of $\qq$ above which $B$ is ramified.
\end{itemize}

We will write ${\rm det}_{B/F}$ resp. ${\rm tr}_{B/F}$ for the reduced
norm resp. trace of $B/F$, and define ${\rm det}_{B/\qq}={\rm det}_{F/\qq}\circ{\rm det}_{B/F}$
and ${\rm tr}_{B/\qq}={\rm tr}_{F/\qq}\circ{\rm tr}_{B/F}$.
Pick an involution $\ast$ of the second kind on $B$, which we assume
is positive (i.e. ${\rm tr}_{B/\qq}(xx^{\ast})>0$ for all
nonzero $x\in B$). Pick $\beta\in B^{\ast=-1}$; the pairing
\[
(x,y)={\rm tr}_{B/\qq}(x\beta y^{\ast})
\]
is an alternating $\ast$-Hermitian pairing $B\times B \ra \qq$,
and we define a new involution of the second kind by $x^{\#}=\beta x^{\ast}\beta^{-1}$.
We have that
\[
(bxc,y)={\rm tr}_{B/\qq}(bxc\beta y^{\ast})={\rm tr}_{B/\qq}(xc\beta y^{\ast}b)={\rm tr}_{B/\qq}(x\beta(\beta^{-1}c\beta)y^{\ast}b)=
\]
\[
=(x,((\beta^{-1}c\beta)y^{\ast}b)^{\ast})=(x,b^{\ast}yc^{\#})
\]
as $(\beta^{-1})^{\ast}=(\beta^{\ast})^{-1}=-\beta^{-1}$. We let
$G/\qq$ be algebraic group whose $R$-points ($R$ any $\qq$-algebra)
are
\[
G(R)=\left\{ (g,\lambda)\in(B^{op}\otimes_{\qq}R)^{\times}\times R^{\times}\mid gg^{\#}=\lambda\right\}
\]
When $x$ is a place of $\qq$ that splits as $yy^{c}$ in $F$,
$y$ induces an isomorphism
\[
G(\qq_{x})\cong(B^{op}\otimes_{\qq}F_{y})^{\times} \times \qq_{x}^{\times}
\]
In particular, we get an isomorphism
\[
G(\Qp)=(B_{u}^{op})^{\times}\times\Qp^{\times}
\]
where we have written $B_{u}^{op}=B^{op}\otimes_{F}F_{u}$.

We assume that (see \cite[Lemma I.7.1]{HT})
\begin{itemize}
\item if $x$ is a rational prime which does not split in $F$, then $G$
is quasi-split at $x$;
\item the pairing $(-,-)$ on $B\otimes_{\qq}\R$ has invariants
$(1,n-1)$.
\end{itemize}

We will fix a maximal order $\Lambda_{u}=\mathcal{O}_{B_{u}}$ in
$B_{u}$. The pairing $(-,-)$ gives a perfect duality between $B_{u}$
and $B_{u^{c}}$ and we define $\Lambda_{u}^{\vee}\subseteq B_{u^{c}}$
to be the dual of $\Lambda_{u}\subseteq B_{u}$. Then
\[
\Lambda=\Lambda_{u}\oplus\Lambda_{u}^{\vee}\subseteq B\otimes_{\qq}\Qp
\]
is a $\Zp$-lattice in $B\otimes_{\qq}\Qp$
and $(-,-)$ restricts to a perfect pairing $\Lambda \times \Lambda \ra \Zp$.
There is a unique $\ZZ_{(p)}$-order $\mc{O}_{B}$ of
$B$ such that $\mc{O}_{B}^{\ast}=\mc{O}_{B}$ and $\mc{O}_{B,u}=\mc{O}_{B_{u}}$
(where $\mc{O}_{B,u}=\mc{O}_{B}\otimes_{\mc{O}_{F,(u)}}\mc{O}_{F_{u}}$,
$\mc{O}_{F,(u)}$ being the algebraic localization of $\mc{O}_{F}$
at $u$). The stabilizer of $\Lambda$ in $G(\Qp)$ is
$(\mc{O}_{B,u}^{op})^{\times} \times \Zp^{\times}$.
Fix an isomorphism $\mc{O}_{B,u}\cong M_{n}(\mc{O}_{F_{u}})=M_{n}(\Zp)$.
By composing it with the transpose map $-^{t}$ we get an isomorphism
$\mc{O}_{B,u}^{op}\cong M_{n}(\Zp)$. If we let $\epsilon\in M_{n}(\ZZ)$
denote the idempotent whose $(i,j)$th entry is $1$ if $i=j=1$ and
$0$ otherwise, as well as the corresponding idempotent in $\mc{O}_{B,u}^{op}$
(using our isomorphism), we get an isomorphism
\[
\epsilon\Lambda_{u}\cong(\mc{O}_{F_{u}}^{n})^{\vee}
\]
We give $\epsilon\Lambda_{u}$ the basis $e_{1}$,...,$e_{n}$ corresponding
to the standard dual basis of $(\mc{O}_{F_{u}}^{n})^{\vee}$.
The induced isomorphism $B_{u}^{op}\cong M_{n}(F_{u})=M_{n}(\Qp)$
gives us an isomorphism
\[
G(\Qp)\cong{\rm GL}_{n}(F_{u})\times\Qp^{\times}={\rm GL}_{n}(\Qp)\times\Qp^{\times}
\]

\subsection{Shimura varieties and integral models}

The main reference for this section is \cite[\S III]{HT}. Define
\[
C=(B^{op}\otimes_{\qq}\R) \times \R
\]
Then there is a unique conjugacy class $\mf{H}$ of $\mathbb{R}$-algebra
homomorphisms
\[
h\,:\,\C \ra C
\]
such that $h(\bar{z})=h(z)^{\ast}$ and the pairing $u,v\mapsto(u,v.h(i))$
is positive or negative definite and symmetric. The pair $(G,\mf{H})$
is Shimura datum whose reflex field is $F$ with respect to the unique
infinite place of $F$. For any compact open subgroup $U\sub G(\mb{A}^{\infty})$
let us write $Sh_{U}$ for the canonical model of the Shimura variety
\[
G(\qq)\backslash G(\mb{A}^{\infty}) \times \mf{H}/U
\]
Following \cite[\S III.1]{HT}, we define a contravariant functor
$X_{U}$ from locally Noetherian schemes over $F$ to sets by letting
$X_{U}(S)$, for any connected locally Noetherian $F$-scheme $S$,
be the set of equivalence classes of quadruples $(A,\lambda,i,\bar{\eta})$
where

\begin{itemize}
\item $A$ is an abelian scheme over $S$ of dimension $n^{2}$;
\item $\lambda\,:\, A \ra A^{\vee}$ is a polarisation;
\item $i\,:\, B \inj {\rm End}_{S}(A) \otimes_{\ZZ}\qq$;
is an algebra homomorphism such that $(A,i)$ is compatible (\cite[p. 90]{HT}) and $\lambda\circ i(b)=i(b^{\ast})^{\vee}\circ\lambda$ for
all $b\in D$;
\item $\bar{\eta}$ is a $U$-level structure on $A$.
\end{itemize}

Two quadruples are equivalent if there is a quasi-isogeny between
the corresponding abelian schemes preserving the structure. See \cite[p. 91]{HT} for more details. When $U$ is neat, $X_{U}$ is representable
by a smooth projective variety over $F$, which we will also denote
by $X_{U}$, and above it we have an abelian scheme $A_{U}$ fitting
in a universal quadruple $(A_{U},\lambda_{U},i_{U},\bar{\eta}_{U})$.
We have
\begin{equation}
X_{U}=Sh_{U}\label{eq: disjoint union}
\end{equation}
because ${\rm ker}^{1}(\qq,G)$ that complicates the
situation in general vanishes since $F^{+}=\qq$ (since the center of $G$ is ${\rm Res}^{F}_{\qq}\mb{G}_{m}$ which is a cohomologically trivial torus; cf. the discussion in \cite[p. 319-322]{Hid1}). By varying
$U$, we get a projective system $(X_{U})_{U}$ of varieties where
the morphisms between are finite, surjective and \'etale and similarly
for the $A_{U}$.

Next we recall the integral models defined in \cite[\S III.4]{HT}.
Put, for $m \in \ZZ_{\geq 0}$,
\[
U_{p}(m)=K_{n}(m) \times \Zp^{\times} \sub{\rm GL}_{n}(\Qp) \times \Qp^{\times}=G(\Qp)
\]
where $K_{n}(m)\sub {\rm GL}_{n}(\Zp)$ is the subgroup
of matrices reducing to the identity modulo $p^{m}$. For $U^{p}\sub G(\mb{A}^{p,\infty})$,
we denote the product $U^{p}\times U_{p}(m)\sub G(\mb{A}^{\infty})$
by $U(m)$. We will denote $U(0)$ by $U$. Define a contravariant
functor $\mc{X}_{U(m)}$ from locally Noetherian $\Zp$-schemes
to sets by letting $\mc{X}_{U(m)}(S)$, for any connected locally
Noetherian $\Zp$-scheme $S$, be the set of equivalence
classes of quintuples $(A,\lambda,i,\bar{\eta}^{p},\alpha)$ where

\begin{itemize}
\item $A/S$ is an abelian scheme of dimension $n^{2}$;
\item $\lambda\,:\, A \ra A^{\vee}$ is a prime-to-$p$ polarisation;
\item $i\,:\,\mc{O}_{B} \inj {\rm End}_{S}(A) \otimes_{\ZZ} \ZZ_{(p)}$
is an algebra homomorphism such that $(A,i)$ is compatible and $\lambda\circ i(b)=i(b^{\ast})^{\vee}\circ\lambda$
for all $b\in\mc{O}_{B}$;
\item $\bar{\eta}^{p}$ is a $U^{p}$-level structure on $A$;
\item $\alpha\,:\, p^{-m}\epsilon \Lambda_{u}/ \epsilon \Lambda_{u} \ra \mc{G}_{A}[p^{m}](S)=\epsilon A[u^{m}](S)$
is a Drinfeld $p^{m}$-level structure.
\end{itemize}

Equivalence here is as before, but using prime-to-$p$ quasi-isogenies.
See \cite[p.109]{HT} for more details and the proof of the representability
of $\mc{X}_{U(m)}$ when $U(m)$ is neat. We will also denote the
representing scheme over $\Zp=\mc{O}_{F,u}$ by $\mc{X}_{U(m)}$.
Its generic fibre is $X_{U(m)/\Qp}$ and we will denote
its special fibre by $\ol{X}_{U(m)}$. Similarly we will let
$\mc{A}_{U(m)}$ denote the universal abelian variety over $\mc{X}_{U(m)}$;
its generic fibre is $A_{U(m)/\Qp}$ and its special fibre
will be denoted by $\ol{A}_{U(m)}$. We will let $\mc{G}$
denote the Barsotti-Tate group $\epsilon \mc{A}_{U(m)}[u^{\infty}]$;
over any base in which $p$ is nilpotent it is one-dimensional and
compatible of height $n$ (see \cite[\S II]{HT} for terminology).
$\mc{G}$ defines a stratification $\left(\ol{X}_{U(m),[h]}\right)_{0\leq h\leq n-1}$
of $\ol{X}_{U(m)}$ where $\ol{X}_{U(m),[h]}$ is the
closed reduced subscheme of $\ol{X}_{U(m)}$ whose geometric
points $s$ are those for which the maximal \'etale quotient of $\mc{G}_{s}$
has height $\leq h$. In the case $m=0$, each $\ol{X}_{U,[h]}$
is smooth (see remark near the end of \cite[p. 113]{HT}; we will
not need this). We let $\ol{X}_{U(m),(h)}=\ol{X}_{U(m),[h]}-\ol{X}_{U(m),[h-1]}$
for $h\geq1$ and $\ol{X}_{U(m),(0)}=\ol{X}_{U(m),[0]}$;
these are smooth of pure dimension $h$ for arbitrary $m$ (\cite[Corollary III.4.4]{HT}). The stratification $\ol{X}_{U(m)}=\coprod_{h=0}^{n-1}\ol{X}_{U(m),(h)}$
is the Newton stratification of $\ol{X}_{U(m)}$. Denote by
$\mc{G}^{(h)}$ the restriction of $\mc{G}$ to $\ol{X}_{U(m),(h)}$.
All these varieties etc. fit into projective systems as we vary $U^{p}$
and $m$, with transition maps being finite, flat and surjective.
If we keep $m$ fixed and only vary $U^{p}$ the maps are \'etale as well.

\subsection{Semistable integral models of Iwahori level}

The reference for this section is \cite[\S 3]{TY}. Let $B_{n}$ denote
the standard upper triangular Borel subgroup of $\GL_{n}$ and
let $I_{n}$ denote the corresponding Iwahori subgroup of $\GL_{n}(\Zp)$
defined as the preimage of $B_{n}(k_{u})=B_{n}(\Fp)$ under
the reduction map $\GL_{n}(\Zp) \ra \GL_{n}(\Fp)$
(here $k_{u}$ denotes the residue field of $F_{u}$). We set, for
any compact open subgroup $U^{p} \sub G(\mb{A}^{p,\infty})$,
\[
Iw_{p}=I_{n} \times \Zp^{\times} \sub \GL_{n}(\Qp) \times \Qp^{\times} = G(\Qp)
\]
\[
Iw=U^{p}\times Iw_{p}\sub G(\mb{A}^{\infty})
\]
Define a contravariant functor $\mc{X}_{Iw}$ from locally Noetherian
$\Zp$-schemes to sets by letting $\mc{X}_{Iw}(S)$,
for any connected locally Noetherian $\Zp$-scheme $S$,
be the set of equivalence classes of quintuples $(A,\lambda,i,\bar{\eta}^{p},\mc{C})$
where $(A,\lambda,i,\bar{\eta}^{p})$ is as for $\mc{X}_{U}$
and $\mc{C}$ is a chain of isogenies ( $\mc{G}_{A}=\epsilon A[u^{\infty}]$):
\[
\mc{G}_{A}=\mc{G}_{0}\ra \mc{G}_{1} \ra ... \ra \mc{G}_{n} = \mc{G}_{A}/\mc{G}_{A}[p]
\]
of compatible Barsotti-Tate groups, each of degree $p=\#k_{u}$, with
composite equal to the canonical map $\mc{G}_{A} \ra \mc{G}_{A}/\mc{G}_{A}[p]$.
Equivalently, we may view it as a flag
\[
0=\mc{C}_{0} \sub \mc{C}_{1} \sub ... \sub \mc{C}_{n} = \mc{G}_{A}[p]
\]
of finite flat subgroup schemes, with each successive quotient of
degree $p$. $\mc{X}_{Iw}$ is representable by a scheme over
$\Zp$ which we will also denote by $\mc{X}_{Iw}$;
its generic fibre is $X_{Iw/\Qp}$ and we will denote its
special fibre by $Y_{Iw}$. We let $\mc{A}_{Iw}$ denote the
universal abelian variety over $\mc{X}_{Iw}$; its generic fibre
is $A_{Iw/\Qp}$ and we denote by $\ol{A}_{Iw}$
its special fibre. Furthermore, we let
\[
\mc{G}=\mc{G}_{0} \ra \mc{G}_{1} \ra ... \ra \mc{G}_{n}
\]
denote the universal chain of isogenies on $\mc{X}_{Iw}$, and
let $Y_{Iw,i}$ denote the closed subscheme of $Y_{Iw}$ over which
$\mc{G}_{i-1} \ra \mc{G}_{i}$ has connected kernel.
Here $\mc{G}=\epsilon \mc{A}_{Iw}[u^{\infty}]$ by abuse
of notation; since this $\mc{G}$ is the pullback of the previous
$\mc{G}$ via the natural map $\mc{X}_{Iw} \ra \mc{X}_{U}$
we hope that this should not cause any confusion. We have the following:

\begin{prop}
(\cite[Proposition 3.4]{TY})

1) $\mc{X}_{Iw}$ is regular, has pure dimension $n$ and the
natural map $\mc{X}_{Iw} \ra \mc{X}_{U}$ is finite
and flat.

2) Each $Y_{Iw,i}$ is smooth over $\Fp$ of pure dimension
$n-1$, $Y_{Iw}=\bigcup_{i}Y_{Iw,i}$ and for $i\neq j$ $Y_{Iw,i}$
and $Y_{Iw,j}$ have no common connected component. In particular,
$\mc{X}_{Iw}$ has strictly semistable reduction and
for each $S \sub \left\{ 1,...,h\right\} $, $Y_{Iw,S}=\bigcap_{I\in S}Y_{Iw,i}$
is smooth of pure dimension $n-\#S$.
\end{prop}

We write
\[
Y_{Iw,S}^{0}=Y_{Iw,S}-\bigcup_{T \supsetneq S}Y_{Iw,T}
\]
and we let $\ol{A}_{Iw,S}$ and $\mc{G}_{S}$ (resp. $\ol{A}_{Iw,S,0}$
and $\mc{G}_{S}^{0}$) denote the restriction of $\ol{A}_{Iw}$
resp. $\mc{G}$ to $Y_{Iw,S}$ (resp. $Y_{Iw,S}^{0}$). When
$S$ is singleton $\left\{ i\right\} $ we will write the subscript
as $i$ instead of $\left\{ i\right\} $. All varieties defined here
fit into projective systems as we vary $U^{p}$, and the transition
maps are finite, \'etale and surjective. Finally, note that under the
natural map $Y_{Iw}\ra \ol{X}_{U}$, $\bigcup_{\#S=n-h}Y_{Iw,S}^{0}$
is the inverse image of $\ol{X}_{U,(h)}$, in particular we
have maps $Y_{Iw,S}^{0} \ra \ol{X}_{U,(n-\#S)}$ for all
$S$. We also have a forgetful map $\mc{X}_{U(1)} \ra \mc{X}_{Iw}$.
It is given by sending a quintuple $(A,\lambda,i,\bar{\eta}^{p},\alpha)$
to the quintuple $(A,\lambda,i,\bar{\eta}^{p},\mc{C})$, where
$\mc{C}$ is the flag of subgroup schemes given by letting $\mc{C}_{i}$
be the unique finite flat subgroup scheme of $\mc{G}_{A}[u]$
that has $\alpha(M_{i})$ as a full set of sections of $\mc{C}_{i}$,
where $M_{i}$ is the subspace of $p^{-1} \epsilon \Lambda_{u}/\epsilon\Lambda_{u}$
generated by $p^{-1}e_{1}$,...,$p^{-1}e_{i}$ (see \cite[Lemma
II.2.4]{HT}).

\subsection{Igusa varieties of the first kind\label{sub:Igusa-varieties}}

The main reference for this section is \cite[\S 4]{HT}. Let $s=\Spec\, \ol{\Fp}$
be a geometric point of $\ol{X}_{U,(n-1)}$. As a Barsotti-Tate
$\mc{O}_{B,p}=\mc{O}_{B} \otimes_{\ZZ_{(p)}}\Zp$-module,
$\ol{A}_{U,(n-1)}[p^{\infty}]_{s}$ decomposes as
\[
\ol{A}_{U,(n-1)}[u^{\infty}]_{s} \times \ol{A}_{U,(n-1)}[(u^{c})^{\infty}]_{s}
\]
with $\ol{A}_{U,(n-1)}[(u^{c})^{\infty}]_{s}=(\ol{A}_{U,(n-1)}[u^{\infty}]_{s})^{\vee}$
(Cartier dual), $\ol{A}_{U,(n-1)}[u^{\infty}]_{s}=(\mc{G}_{s}^{(n-1)})^{n}$
as Barsotti-Tate $M_{n}(\Zp)$-modules and $\mc{G}_{s}^{(n-1)} \cong \mu_{p^{\infty}} \times (\Qp/\Zp)^{n-1}$
as Barsotti-Tate groups. The Igusa variety of the first kind $Ig_{U^{p},m}/\ol{X}_{U,(n-1)}$
is defined to be the moduli space for isomorphisms
\[
\alpha^{et}\,:\,(p^{-m}\Zp/\Zp)^{n-1} \tilde{\ra}\mc{G}^{et}[u^{m}].
\]
See \cite[p. 121]{HT}. The forgetful morphism $Ig_{U^{p},m} \ra \ol{X}_{U,(n-1)/\bar{k}}$
is a Galois cover with Galois group $\GL_{n-1}(\mc{O}_{F_{u}}/u^{m})=\GL_{n-1}(\ZZ/p^{m})$.
We think of $\GL_{n-1}$ as a factor of the Levi subgroup $\GL_{1} \times \GL_{n-1}$
of $\GL_{n}$ of block diagonal matrices with square blocks of
side lengths $1$ and $n-1$. The $Ig_{U^{p},m}$ fit together in
a projective system when varying $U^{p}$ and $m$, with the transition
maps being finite and \'etale.

We also need to recall (a special case of) the Iwahori-Igusa variety
of the first kind $Ig_{Iw}$ defined in the beginning of
\cite[\S 4]{TY} (there denoted by $I_{U}^{(n-1)}$). It is defined as the
moduli space of chains of isogenies
\[
\mc{G}^{(n-1),et}=\mc{G}_{1} \ra \mc{G}_{2} \ra ... \ra \mc{G}_{n} = \mc{G}^{(n-1),et}/\mc{G}^{(n-1),et}[u]
\]
over $\ol{X}_{U,(n-1)}$. The natural map $Ig_{Iw} \ra \ol{X}_{U,(n-1)}$
is finite and \'etale, and the natural map $Ig_{U^{p},1} \ra Ig_{Iw}$
is Galois with Galois group $B_{n-1}(\Fp)$. Moreover,
there is a natural map $Y_{Iw,1}^{0} \ra Ig_{Iw}$ defined
by taking \'etale quotients in the chain of isogenies on $Y_{Iw,1}^{0}$;
by \cite[Lemma 4.1]{TY} it is a finite map that is bijective on
geometric points. In fact it is an isomorphism, the map in the other
direction coming from augmenting the chain of isogenies
\[
\mc{G}^{(n-1),et}=\mc{G}_{1} \ra \mc{G}_{2} \ra ...\ra \mc{G}_{n} = \mc{G}^{(n-1),et}/ \mc{G}^{(n-1),et}[w]
\]
with the $p$-th power Frobenius morphism $\mc{G}^{(n-1)} = \mc{G}_{0} \overset{Fr}{\ra} \mc{G}_{1} = \mc{G}^{(n-1),(q)} = \mc{G}^{(n-1),et}$
(note that this is special to $Y_{Iw,1}^{0}$; it does not hold for
$Y_{Iw,1}^{i}$ when $i \neq 0$). To summarize, we have that the Igusa
varieties $Ig_{U^{p},m}$ form a projective system of Galois covers
of $Y_{Iw,1}^{0}$, and the Galois group of $Ig_{U^{p},1} \ra Y_{Iw,1}^{0}$
is $B_{n-1}(\Fp)$.

\section{Hecke Actions and Automorphic Forms\label{sec:Hecke-Actions}}

In this section we will first relate our Igusa varieties to the Shimura
varieties and then define automorphic vector bundles and overconvergent
$F$-isocrystals on them. We will then go on to discuss the various Hecke algebras acting on each finite
level and the commutative subalgebra (the Atkin-Lehner ring) that
acts on automorphic forms and that we will use in our control theorems. Finally, the last section will discuss $p$-divisibility of the $U_{p}$-operator in this context, following Hida.

\subsection{\label{sub:Hecke-actions}Hecke actions on the Shimura varieties
and Igusa varieties}

The action of $G(\mb{A}^{\infty})$ on the system $(\ol{X}_{U(m)})_{U^{p},m}$
is described on p.109 of \cite{HT}. On \cite[p.116]{HT} a decomposition
\[
\ol{X}_{U(m),(h)}=\coprod_{M} \ol{X}_{U(m),M}
\]
coming from the Drinfeld level structure (see \cite[Lemma II.2.1(5)]{HT}) is defined. Here the $M$ range over free rank $n-h$ direct summands
of $p^{-n}\epsilon \Lambda_{u}/ \epsilon \Lambda_{u}$. Now let $M$
be a free rank $n-h$ direct summand of $\epsilon \Lambda_{u}$, and let
$P_{M} \sub {\rm Aut}(\epsilon\Lambda_{u})$ denote the parabolic
subgroup stabilizing $M$. If we set $\ol{X}_{U(m),M}:= \ol{X}_{U(m),p^{-m}M/M}$,
then the $\ol{X}_{U(m),M}$ form a projective system with an
action of the group
\[
G_{M}(\mb{A}^{\infty}) := G(\mb{A}^{p,\infty}) \times P_{M}(\Qp) \times \Qp^{\times}
\]
From now on, we fix $M=\left\langle e_{1}\right\rangle \sub \epsilon\Lambda_{u}$.
If we look at the forgetful map $\ol{X}_{U(1)} \ra Y_{Iw}$,
we see that the preimage of $Y_{Iw,1}^{0}$ consists of those $(A,\lambda,i,\bar{\eta}^{p},\alpha)$
for which $\alpha(p^{-1}e_{1})=0$, i.e. the preimage is $\ol{X}_{U(1),M}$.
Thus we see that $\ol{X}_{U(1),M} \ra Y_{Iw,1}^{0}$
is Galois with Galois group $B_{n-1}(\Fp)$ and that the
$(\ol{X}_{U(m),M})_{m\geq1}$ form a tower above $Y_{Iw,1}^{0}$
(as well as above $\ol{X}_{U,(n-1)}=\ol{X}_{U,M}$).

Next let us consider Igusa varieties. We follow \cite[p. 122
ff]{HT} (but note our choices of $h=n-1$ and $M=\left\langle e_{1}\right\rangle $).
The projective system $Ig_{U^{p},m}$ carries an action of
\[
G(\mb{A}^{p,\infty}) \times (\ZZ \times \GL_{n-1}(\Qp))^{+} \times \Qp^{\times}
\]
where $(\ZZ \times \GL_{n-1}(\Qp))^{+} := \left\{ (c,g)\in\ZZ \times \GL_{n-1}(\Qp) \mid p^{-c}g\in M_{n-1}(\Zp)\right\} $.
We denote by $j$ the surjection $\epsilon\Lambda_{u} \ra \Zp^{n-1}$
with kernel $M$ given by identifying $\Zp^{n-1}$
with the direct summand $\left\langle e_{2},...,e_{n}\right\rangle $.
Then define a homomorphism $j_{\ast}\,:\, P_{M}(\Qp) \ra \ZZ \times \GL_{n-1}(\Qp)$
by $g\mapsto({\rm ord}_{p}(g|_{M}),j\circ g\circ j^{-1})$ which induces
a homomorphism
\[
G_{M}(\mb{A}^{\infty}) \ra G(\mb{A}^{p,\infty}) \times \ZZ \times \GL_{n-1}(\Qp) \times \Qp^{\times}
\]
that we will also denote by $j_{\ast}$. There is a morphism
\[
j^{\ast}\,:\, Ig_{U^{p},m} \ra \ol{X}_{U(m),M}
\]
defined on \cite[p. 124]{HT}, given by sending $(A,\lambda,i,\bar{\eta}^{p},\alpha^{et})$
to $(A^{(p^{m})},\lambda^{(p^{m})},i,\bar{\eta}^{p},\alpha)$, where
$-^{(p^{m})}$ denote twisting by the $p^{m}$-power Frobenius $F^{m}$
and $\alpha=F^{m} \circ \alpha^{et}\circ j$. In fact, $j^{\ast}$ is
an isomorphism. The map $j^{\ast}Fr^{m}$ is therefore a finite purely
inseparable morphism and for $g\in G_{M}(\mb{A}^{\infty})$ such
that $j_{\ast}(g) \in G(\mb{A}^{p,\infty}) \times (\ZZ \times \GL_{n-1}(\Qp))^{+} \times \Qp^{\times}$,
the diagram
\begin{equation}
\xymatrix{Ig_{U^{p},m}\ar[r]^{j_{\ast}(g)}\ar[d]^{j^{\ast}Fr^{m}} & Ig_{U^{p},m^{\prime}}\ar[d]^{j^{\ast}Fr^{m^{\prime}}}\\
\ol{X}_{U(m),M}\ar[r]^{g} & \ol{X}_{U(m^{\prime}),M}
}
\label{eq: compatibility}
\end{equation}
commutes (where $m$ and $m^{\prime}$ are chosen so that one may
define $g$ and $j_{\ast}(g)$). It follows that this action extends
to an action of $G(\mb{A}^{p,\infty}) \times \ZZ \times \GL_{n-1}(\Qp) \times \Qp^{\times}$
on cohomology.

\subsection{Representations, motives and sheaves}

We wish to define various sheaves on our varieties that we want to
work with. Before we can do so we need to recall some basic facts
about the finite dimensional algebraic representations of our group
$G$. We have $G(\R)={\rm GU}(n-1,1)$. As is well known,
the finite dimensional representations of $\GL_{n}(\C)$
are parametrised by integers $k_{1}\geq...\geq k_{n}$, which are
dominant weights for the diagonal torus $T_{n}$ with respect to the
upper triangular Borel $B_{n}$. We parametrise complex representations
of $GU(n-1,1)$ by $n+1$-tuples of integers $(k_{1},...,k_{n},w)$
where $k_{1}\geq...\geq k_{n}$ parametrises an irreducible representation
of $U(n-1,1)$ and $w \equiv \sum k_{i}\,{\rm mod}\,2$ parametrises
a character of the split part of the center of $GU(1,n-1)$.

Algebraic representations $\xi$ of $GU(n-1,1)$ define sheaves $V_{et}(\xi)$, $V_{dR}(\xi)$, $V_{sing}(\xi)$ and $V^{\dg}(\xi)=V_{rig}(\xi)$ in various cohomology theories (\'etale, de Rham or singular on the
generic fibre, and rigid cohomology on the special fibre); see \cite[\S III.2]{HT}. We apologize here for the double notation $V^{\dg}(\xi)=V_{rig}(\xi)$; we prefer $V^{\dg}(\xi)$ and will use it throughout the text except for in the paragraph below when $V_{rig}(\xi)$ will be more convenient. It will be convenient to define these sheaves
as realizations of motives, defined using self-products of the universal
abelian variety together with an idempotent. Let us recall the construction
from \cite[\S III.2]{HT} and \cite[\S 2]{TY}. Given $\xi$,
there are integers $t_{\xi}$ and $m_{\xi}$ and an idempotent $a_{\xi}$
such that
\begin{itemize}
\item $a_{\xi}R^{m_{\xi}}\pi_{m_{\xi},\ast}\qq_{?}(t_{\xi})\cong V_{?}(\xi)$;
\item $a_{\xi}H_{?}^{j}(A,\qq_{?}(t_{\xi})=H_{?}^{j-m_{\xi}}(X,V_{?}(\xi))$ for
$j\geq m_{\xi}$ and $0$ otherwise.
\end{itemize}
for any of the above mentioned cohomology theories, where we have
used $X$ for some smooth subvariety of a Shimura
variety (in characteristic $0$ or $p$), $A$ for the universal abelian variety over $X$, $\qq_{?}$ for the constant sheaf in the relevant cohomology
theory and $\pi_{m_{\xi}}\,:\, A^{m_{\xi}}\ra X$ denotes
the projection; here $?\in \{et,dR,sing,rig\}$. $(A^{m_{\xi}},a_{\xi},t_{\xi})$ defines a classical
motive that we will denote by $V(\xi)$ (see e.g. \cite{Sch}) and the $V_{?}(\xi)$ are its realizations. Moreover $a_{\xi}$ commutes with the
action of $G(\mb{A}^{\infty})$.

We will be interested in the overconvergent $F$-isocrystals $V^{\dg}(\xi)=a_{\xi}H_{rig}^{m_{\xi}}(\ol{A}_{Iw,1,0}/Y_{Iw,1}^{0})$
on $Y_{Iw,1}^{0}$. Using the frame (\cite[Definition 3.1.5]{LeSt})
\[
Y_{Iw,1}^{0}\sub Y_{Iw} \sub \wh{\mc{X}}_{Iw}
\]
(the hat for the completion along the special fibre) we may construct
the underlying overconvergent isocrystal by considering the de Rham
realisation $V_{dR}(\xi)$ on $X_{Iw/\Qp}$; analytifying
it to the Raynaud generic fibre $X_{Iw}^{rig}$ of $\wh{\mc{X}}_{Iw}$
(which agrees with the Tate analytification of $X_{Iw/\Qp}$
by properness) and finally applying Berthelot's $j_{Y_{Iw,1}^{0}}^{\dg}$-functor
(\cite[\S 5.1]{LeSt} ; it is probably easiest to use \cite[Proposition 5.1.12]{LeSt}
as the definition).

Next, consider the parabolic $Q=Q_{n-1,1} \times \C^{\times}$
of $G(\C) \cong \GL_{n}(\C) \times \C^{\times}$
where $Q_{n-1,1}$ is the standard $(n-1,1)$ parabolic of $\GL_{n}(\C)$
with last row of the form $(0\,...\,0\,\ast)$. Weights $(k_{1},...,k_{n},w)$
for $G$ are dominant for the Levi $L_{Q}$ of $Q$ with respect to
$B_{n}\cap L_{Q}$ if and only if $k_{1}\geq...\geq k_{n-1}$. The
automorphic vector bundle construction of Milne and Harris (\cite{Mil},
see also \cite[p. 101]{HT} for the construction over $\C$)
produces out of finite-dimensional algebraic representations $\mu$
of $Q$ vector bundles $W(\mu)$ on $X_{Iw}(\C)$ that descend
to $F$, so we may base change them to $F_{u}=\Qp$
(this uses that $G$ is split over $F$). We define $W^{\dg}(\mu)=j_{Y_{Iw,1}^{0}}^{\dg}W(\mu)$
using the frame above. If $sp\,:\, X_{Iw}^{rig} \ra Y_{Iw}$
denotes the specialization map, let us put $X_{Iw}^{ord}=sp^{-1}(Y_{Iw,1}^{0})$.
We remark that this is not the full ordinary locus in $X_{Iw}^{rig}$,
but rather what could be called the ordinary-multiplicative locus.

We will also need the above discussion for higher levels. Let $m\geq 1$. Here we consider the frame
$$ \ol{X}_{U(m),M} \sub \ol{X}_{U(m)} \sub \wh{\mc{X}}_{U(m)}. $$
There is a specialization map $sp\, :\, X_{U(m)}^{rig} \ra \ol{X}_{U(m)}$ and we define $X_{U(m)}^{ord}:= sp^{-1}(\ol{X}_{U(m),M})$. Again we get an overconvergent $F$-isocrystal $V^{\dg}(\xi)=a_{\xi}H_{rig}^{m_{\xi}}(\ol{A}_{U(m),M}/\ol{X}_{U(m),M})$ on $\ol{X}_{U(m),M}$, whose underlying overconvergent isocrystal may be constructed by applying $j^{\dg}_{\ol{X}_{U(m),M}}$ to $V_{dr}(\xi)$ on $X_{U(m)}^{rig}$. We also have vector bundles $W(\mu)$ on $X_{U(m)/\mb{Q}_{p}}$ and we put $W^{\dg}(\mu)=j^{\dg}_{\ol{X}_{U(m),M}}W(\mu)$.

\begin{defn}
Let $m\geq 1$.
\begin{enumerate}
\item An automorphic form of weight $\mu$, tame level $U^{p}$ and Iwahori
level (resp. level $m$) at $p$ is an element of $H^{0}(X_{Iw}^{rig},W(\mu))$ (resp. $H^{0}(X_{U(m)}^{rig},W(\mu))$).

\item An overconvergent automorphic form of weight $\mu$, tame level
$U^{p}$ and Iwahori level (resp. level $m$) at $p$ is an element of $H^{0}(X_{Iw}^{rig},W^{\dg}(\mu))$ (resp. $H^{0}(X_{U(m)}^{rig},W^{\dg}(\mu))$).

\item A $p$-adic automorphic form of weight $\mu$, tame level $U^{p}$ and Iwahori level (resp. level $m$) at $p$
is an element of $H^{0}(X_{Iw}^{ord},W(\mu))$ (resp. $H^{0}(X_{U(m)}^{ord},W(\mu))$).
\end{enumerate}
\end{defn}

\begin{rem}
We will usually suppress any mention of the tame level. The tame level is assume to be fixed throughout this paper;
it plays no explicit role.
\end{rem}
We have natural inclusions 
$$H^{0}(X_{Iw}^{rig},W(\mu)) \sub H^{0}(X_{Iw}^{rig},W^{\dg}(\mu)) \sub H^{0}(X_{Iw}^{ord},W(\mu))$$
and 
$$H^{0}(X_{U(m)}^{rig},W(\mu)) \sub H^{0}(X_{U(m)}^{rig},W^{\dg}(\mu)) \sub H^{0}(X_{U(m)}^{ord},W(\mu)).$$

\subsection{Hecke algebras\label{sub:Hecke-algebras}}

In this section we will define the Hecke algebras that we will be
working with in the next section and recall some results from the
theory of smooth admissible representations of $p$-adic groups. For
any compact open subgroup $V=U^{p}V_{p}\sub G(\mb{A}^{\infty})$,
we will let $\mc{H}_{V}$ denote the full Hecke algebra of smooth,
$\qq$-valued compactly supported bi-$V$-invariant function
on $G(\mb{A}^{\infty})$ (generated over $\qq$ by the
double cosets $VgV$, $g\in G(\mb{A}^{\infty})$ via their characteristic
functions), and we will similarly let $\mc{H}_{U^{p}}$ denote
the full Hecke algebra away from $p$ (which is independent of $V$).
Throughout the rest of the paper $\mc{H}^{p} \sub \mc{H}_{U^{p}}$
will denote a fixed commutative subalgebra, assumed to the full spherical
Hecke algebra at all places for which $U^{p}$ is a hyperspecial maximal
compact subgroup.

Now consider $V=Iw$, corresponding to the compact open $Iw_{p}=I_{n} \times \Zp^{\times} \sub \GL_{n}(\Qp) \times \Qp^{\times}$.
Let $\mc{H}_{Iw_{p}}$ denote the full Hecke algebra over $\qq$
generated by the double cosets $Iw_{p}gIw_{p}$, $g \in \GL_{n}(\Qp) \times \Qp^{\times}$.
We let $\mc{H}_{p}$ denote the subalgebra generated by double
cosets of the forms
\[
U(a_{1},...,a_{n},b)=Iw_{p}(diag(p^{a_{1}},...,p^{a_{n}}),p^{b})Iw_{p}
\]
for integers $a_{1}\geq ... \geq a_{n}$ and $b$. If we forget the
$\Qp^{\times}/\Zp^{\times}$-part this algebra
is sometimes called the Atkin-Lehner ring or the ``dilating'' Hecke
algebra, see e.g. \cite[\S 6.4.1]{BC}. $\mc{H}_{p}$ is commutative
and we have that
\[
U(a_{1},...,a_{n},b)U(a_{1}^{\prime},...,a_{n}^{\prime},b^{\prime})=U(a_{1}+a_{1}^{\prime},...,a_{n}+a_{n}^{\prime},b+b^{\prime}),
\]
so the generators $U(a_{1},...,a_{n})$ form a monoid that we will
denote by $\mc{U}^{-}$. We will be interested in one particular
element of $\mc{H}_{p}$, namely $U(0,-1,...,-1,-1)$, which
we will denote by $U_{p}$. We have a coset decomposition
\[
Iw_{p}(diag(1,p^{-1},...,p^{-1}),p^{-1})Iw_{p}=\coprod_{E\in\Fp^{n-1}}\left(\left(\begin{array}{cc}
1 & p^{-1}E\\
0 & p^{-1}Id_{n-1}
\end{array}\right),p^{-1}\right)Iw_{p}
\]
where we abuse notation and let $E$ also denote an arbitrary lift
of $E$ to $\Zp^{n-1}$. We let $\mc{H}=\mc{H}^{p}\otimes_{\qq}\mc{H}_{p}$;
this is the commutative subalgebra of $\mc{H}_{Iw}$ that we
will use when talking about forms. $\mc{H}_{Iw}$ acts on our
Shimura varieties and universal abelian varieties (including integral
models) of level $Iw$ and their cohomology via the action of $G(\mb{A}^{\infty})$
on the towers; this corresponds to choosing a Haar measures on $G(\mb{A}^{p,\infty})$
and $G(\mb{Q}_{p})$ so that $U^{p}$ and $Iw_{p}$ have measure
$1$.

We may also consider a similar subalgebra at level $m\geq 1$; we will use the same notation to emphasize that our discussions of the (different) higher level cases and the Iwahori case are entirely parallel; this should (hopefully) not cause any confusion. Recall the subgroup $U_{p}(m)=K_{n}(m) \times \Zp^{\times}$. As before we define
$$ U(a_{1},...,a_{n},b) = U_{p}(m)(diag(p^{a_{1}},...,p^{a_{n}}),p^{b})U_{p}(m) $$
for any integers $a_{1}\geq ...\geq a_{n}$ and $b$. Once again the $U(a_{1},...,a_{n},b)$ generate a commutative subalgebra $\mc{H}_{p}$ of the full Hecke algebra for $U_{p}(m)$ and 
$$ U(a_{1},...,a_{n},b)U(a_{1}^{\prime},...,a_{n}^{\prime},b^{\prime})=U(a_{1}+a_{1}^{\prime},...,a_{n}+a_{n}^{\prime},b+b^{\prime}). $$
These assertions (as well as the ones above) follow from \cite[Lemma 4.1.5]{Cas}; here we pick Haar measures such that $U^{p}$ and $U_{p}(m)$ have measure $1$. Again put $\mc{H}=\mc{H}^{p}\otimes_{\Qp}\mc{H}_{p}$ and let $\mc{U}^{-}$ denote the monoid consisting of the $U(a_{1},...,a_{n},b)$. For certain computations we will need to use an auxiliary subgroup that is slightly bigger than $U_{p}(m)$. We let 
$$ K_{n}^{\prime}(m)=\{ g\in \GL_{n}(\Zp) \mid g \equiv diag(\ast,1,...,1)\, {\rm mod}\, p^{m} \} $$
and put $U^{\prime}_{p}(m)=K_{n}^{\prime}(m) \times \Zp^{\times}$. We may once again embed $\mc{U}^{-}$ into the corresponding Hecke algebra; everything goes through exactly as above.

We also need to consider Hecke algebras at $p$ acting on our Igusa
varieties. Let us first focus on the case of Iwahori level. Recall from \S \ref{sub:Hecke-actions} that the monoid
$(\ZZ \times \GL_{n-1}(\Qp))^{+} \times \Qp^{\times}$
acts on $(Ig_{U^{p},m})_{m}$, and that this action extends to an
action of $\ZZ \times \GL_{n-1}(\Qp) \times \Qp^{\times}$
on cohomology. This gives us an action of the Hecke algebra $\mc{H}_{Ig,p}$
generated by the double cosets in
\[
Iw_{Ig}\backslash(\ZZ \times \GL_{n-1}(\Qp) \times \Qp^{\times})/Iw_{Ig}
\]
where $Iw_{Ig}=0 \times I_{n-1} \times \Zp^{\times}$. Put
$\mc{H}_{Ig}=\mc{H}^{p} \otimes_{\qq} \mc{H}_{Ig,p}$.
The compatibility in equation \ref{eq: compatibility} gives us the
following:
\begin{lem}\label{tower}
\label{lem: 4}$\varinjlim H_{rig,c}^{i}(Ig_{U^{p},m},V(\xi))\cong \varinjlim H_{rig,c}^{i}(\ol{X}_{U(m),M},V(\xi))$
as $G_{M}(\mb{A}^{\infty})$-modules, where the group $\varinjlim H_{rig,c}^{i}(Ig_{U^{p},m},V(\xi))$
has the action of $G_{M}(\mb{A}^{\infty})$ given by pulling back
its natural action of $G(\mb{A}^{p,\infty}) \times \ZZ \times \GL_{n-1}(\Qp) \times \Qp^{\times}$
via $j_{\ast}$.
\end{lem}
We embed the monoid $\mc{U}^{-}$ into $\mc{H}_{Ig,p}$
by sending $U(a_{1},...,a_{n},b)$ to $Iw_{Ig}(a_{1},diag(a_{2},...,a_{n}),b)Iw_{Ig}$.
Note that $\mc{H}_{Ig,p}$ is isomorphic to the Hecke algebra
generated by the double cosets in
\[
(\Zp^{\times} \times I_{n-1} \times \Zp^{\times}) \backslash (\Qp^{\times} \times \GL_{n-1}(\Qp) \times \Qp^{\times}) / (\Zp^{\times} \times I_{n-1} \times \Zp^{\times}).
\]
We will need the following lemma in \S \ref{sub:Comp-cl}:
\begin{lem}
\label{lem: BC}Let $\pi$ be an irreducible admissible representation
of $\GL_{n}(\Qp) \times \Qp^{\times}$and consider $\sigma=\pi^{1\times \Zp^{\times}}$. Put $P=P_{M}\times \Qp^{\times}$. Let $P_{M}=L_{M}N_{M}$ be a Levi decomposition with $L_{M}$ a block diagonal Levi. Then the natural map
\[
\pi^{Iw_{p}}=\sigma^{I_{n}} \ra (\sigma_{N_{M}})^{\Zp^{\times}\times I_{n-1}} \otimes \delta_{P}^{-1}
\]
is an isomorphism of $\mc{U}^{-}$-modules. Here $\sigma_{N_{M}}$ denotes
the $N_{M}$-coinvariants (the unnormalized Jacquet module), and $\delta_{P}$
is the modular character with respect to the parabolic $P$.\end{lem}
\begin{proof}
Let $B_{n}=T_{n}U_{n}$ denote the Borel inside $P_{M}$ (notation for this proof
only). Then by \cite[ Proposition 6.4.3]{BC} the natural maps
\[
\sigma^{I_{n}} \ra (\sigma_{U_{n}})^{T_{0}} \otimes \delta_{B_{n}}^{-1},
\]
\[
(\sigma_{N_{M}})^{\Zp^{\times}\times I_{n-1}} \ra ((\sigma_{N_{M}})_{U_{n}\cap L_{M}})^{T_{0}} \otimes \delta_{B_{n}\cap L_{M}}^{-1}
\]
are isomorphisms (and $\mc{U}^{-}$-equivariant); here $T_{0}=T_{n}(\Zp)$
inside $T$. The Lemma now follows since $(\pi_{N_{M}})_{U_{n}\cap L_{M}} \cong \pi_{U_{n}}$
via the natural map, $\delta_{P}^{-1} = \delta_{B_{n}}^{-1}\delta_{B_{n}\cap L_{M}}$,
and the natural map $\sigma^{I_{n}} \ra (\sigma_{U_{n}})^{T_{0}}$ is
the composition of the natural maps $\sigma^{I_{n}} \ra (\sigma_{N_{M}})^{\Zp^{\times} \times I_{n-1}}$
and $(\sigma_{N_{M}})^{\Zp^{\times} \times I_{n-1}} \ra ((\sigma_{N_{M}})_{U_{n}\cap L_{M}})^{T_{0}}$.
\end{proof}
Let us consider the operator $U_{p} \in \mc{U}^{-}$ from above.
The corresponding element inside $\mc{H}_{Ig,p}$ is the double
coset
\[
Iw_{Ig}(0,diag(p^{-1},...,p^{-1}),p^{-1})Iw_{Ig}.
\]
$(0,diag(p^{-1},...,p^{-1}),p^{-1})$ is central, so the double coset acts
by $(0,diag(p^{-1},...,p^{-1}),p^{-1})$. Lemma \ref{lem: BC} tells us that
$U_{p}$ acts on $\pi^{Iw_{p}}$ as $p^{n-1}Iw_{Ig}(0,diag(p^{-1},...,p^{-1}),p^{-1})Iw_{Ig}$
does on $(\pi_{N})^{Iw_{Ig}}$ (with $\pi$ as in the statement of
Lemma). On the other hand, geometrically, we see from Lemma \ref{lem: 4}
that the double coset $U_{p}$, decomposed as
\[
Iw_{p}(diag(1,p^{-1},...,p^{-1}),p^{-1})Iw_{p} = \coprod_{E \in \Fp^{n-1}}x_{E}Iw_{p},
\]
\[
x_{E}=\left(\left(\begin{array}{cc}
1 & p^{-1}E\\
0 & p^{-1}Id_{n-1}
\end{array}\right),p^{-1}\right),
\]
acts the same way as $\coprod_{E}j_{\ast}(x_{E})Iw_{Ig}=p^{n-1}.(0,diag(p^{-1},...,p^{-1}),p^{-1})Iw_{Ig}$.
These observations will be used in \S \ref{sub:Comp-cl}. Note that
similar observations apply to all elements of $\mc{U}^{-}$.

We now discuss the action for higher levels. Since $\ZZ \times \GL_{n-1}(\Qp) \times \Qp^{\times}$ acts on the cohomology of the tower $(Ig_{U^{p},m})_{m}$ we get an action of the Hecke algebra $\mc{H}_{Ig,p}^{(m)}$ generated by the double cosets in 
$$ U_{Ig}(m) \backslash (\mb{Z} \times \GL_{n-1}(\Qp) \times \Qp^{\times}) / U_{Ig}(m) $$
where $U_{Ig}(m)=0 \times K_{n-1}(m) \times \Zp^{\times}$ (here $K_{n-1}(m)$ is the subgroup of matrices in $\GL_{n-1}(\Zp)$ that reduce to the identity modulo $p^{m}$) on cohomology of $Ig_{U^{p},m}$. Put $\mc{H}^{(m)}_{Ig}=\mc{H}^{p} \otimes_{\qq} \mc{H}_{Ig,p}^{(m)}$. Similar to the Iwahori case we embed $\mc{U}^{-}$ into $\mc{H}_{Ig,p}^{(m)}$ by  
$$ U(a_{1},...,a_{n},b) \mapsto U_{Ig}(m)(a_{1},diag(p^{a_{1}},...,p^{a_{n}}),p^{b})U_{Ig}(m).$$
We remark that all coset decomposition written out above in the Iwahori cases works the same for the higher level cases (this follows from \cite[Lemma 1.5.1]{Cas}, which implies that the number a single cosets in the double coset cannot increase). If $\pi$ is smooth admissible representation of $\GL_{n}(\Qp)\times \Qp^{\times}$ we put
$$ \pi^{U_{p}^{\prime}(m)}_{\mc{U}^{-}} := \bigcap_{u\in \mc{U}^{-}}{\rm Im}(u\, :\, \pi^{U_{p}^{\prime}(m)} \ra \pi^{U_{p}^{\prime}(m)}). $$
We remark that this is equal to the object defined before \cite[Lemma 4.1.7]{Cas} and that each $u\in \mc{U}^{-}$ acts invertibly on $\pi^{U_{p}(m)}_{\mc{U}^{-}}$ (these facts follow from \cite[Lemma 4.1.7]{Cas}). If $\sigma$ is smooth admissible representation of $\Qp^{\times} \times \GL_{n-1}(\Qp) \times \Qp^{\times}$ we similarly define 
$$ \sigma_{\mc{U}^{-}}^{U_{Ig}(m)} := \bigcap_{u\in \mc{U}^{-}}{\rm Im}(u\, :\, \sigma^{U_{Ig}(m)} \ra \sigma^{U_{Ig}(m)}); $$
the same remarks apply to this object. We may then state the following higher level version of Lemma \ref{lem: BC}. The proof is the same,  following the proof of \cite[Proposition 6.4.3] {BC} and Lemma \ref{lem: BC}; we sketch it for completeness.

\begin{lem}\label{lem: BC2}
Let $\pi$ be an irreducible admissible representation
of $\GL_{n}(\Qp) \times \Qp^{\times}$and consider $\sigma=\pi^{1\times \Zp^{\times}}$. Put $P=P_{M}\times \Qp^{\times}$. Let $P_{M}=L_{M}N_{M}$ be a Levi decomposition with $L_{M}$ a block diagonal Levi. Then the natural map
\[
\pi^{U_{p}^{\prime}(m)}_{\mc{U}^{-}}=\sigma^{K_{n}(m)}_{\mc{U}^{-}} \ra (\sigma_{N_{M}})^{\Zp^{\times}\times K_{n-1}(m)}_{\mc{U}^{-}} \otimes \delta_{P}^{-1}
\]
is an isomorphism of $\mc{U}^{-}$-modules. Here $\sigma_{N_{M}}$ denotes
the $N_{M}$-coinvariants (the unnormalized Jacquet module), $\delta_{P}$
is the modular character with respect to the parabolic $P$, and we use the notation $-_{\mc{U}^{-}}$ in the obvious way in situations not strictly covered by the remarks above.
\end{lem}

\begin{proof}
As in the proof of Lemma \ref{lem: BC} we write $B_{n}=T_{n}U_{n}$ for the Borel subgroup inside $P_{M}$. We put $T_{0}(m) = T_{n} \cap K_{n}^{\prime}(m) $. The natural map
$$ \sigma_{\mc{U}^{-}}^{K_{n}^{\prime}(m)} \ra (\sigma_{U_{n}})^{T_{0}(m)} $$
is a bijection by \cite[Proposition 4.1.4]{Cas}, and the $\mc{U}^{-}$-equivariance of 
$$ \sigma_{\mc{U}^{-}}^{K_{n}^{\prime}(m)} \ra (\sigma_{U_{n}})^{T_{0}(m)} \otimes \delta_{B_{n}}^{-1} $$
follows by \cite[Lemma 1.5.1, Proposition 4.1.1]{Cas} upon noting that 
$$ U(a_{1},...,a_{n},b)v=\delta_{B}^{-1}(g)\mc{P}_{U_{p}^{\prime}(m)}(gv) $$
where $g=(diag(p^{a_{1}},...,p^{a_{n}}),p^{b})$ and we use the notation of \cite{Cas} (this is an easy computation). The same argument gives us a $\mc{U}^{-}$-equivariant isomorphism
$$ (\sigma_{N_{M}})_{\mc{U}^{-}}^{\Zp^{\times} \times K^{\prime}_{n-1}(m)} \tilde{\ra}\, ((\sigma_{N_{M}})_{U_{n} \cap L_{M}})^{T_{0}(m)} \otimes \delta_{B_{n} \cap L_{M}}^{-1} $$
and to conclude we then argue as in the proof of Lemma \ref{lem: BC}.
\end{proof}

\subsection{Integral models for spaces of $p$-adic automorphic forms and integrality of Hecke operators\label{sub: Hida}}

In this section we we will recall some results and constructions of Hida (\cite{Hid,Hid1}). Our focus is to obtain integrality results for the $U_{p}$-operator. The results are not new and can be dug out of \cite{Hid}. We have nevertheless decided, at the suggestion of a referee, to include some details. Since the results are well known to the experts we will be somewhat brief. The author wishes to thank Vincent Pilloni for useful discussions regarding normalisations of Hecke operators in general.

We start by noting that both $Y_{Iw,1}^{0}$ and the $\ol{X}_{U(m),M}$ are the special fibres of natural formal schemes that we will denote by $\mf{X}_{Iw,1}^{0}$ and $\mf{X}_{U(m),M}$. Indeed, they are open subsets of $Y_{Iw}$ and $\ol{X}_{U(m)}$ respectively, and hence defined open formal subschemes of the completions $\mf{X}_{Iw}$ and $\mf{X}_{U(m)}$ of $\mc{X}_{Iw}$ and $\mc{X}_{U(m)}$, respectively, along $p=0$. They are solutions to moduli problems on the category $Nilp^{ft}_{\Zp}$ of $\Zp$-schemes of finite type on which $p$ is nilpotent, which we now describe briefly. The moduli problem for $\mf{X}_{Iw}$ is obtained by restricting that of $\mc{X}_{Iw}$; the same applies to $\mf{X}_{U(m)}$. To obtain the subfunctor corresponding to $\mf{X}_{Iw,1}^{0}$ one just has to impose that for $S\in Nilp_{\Zp}^{ft}$ and an $S$-point $x$ of $\mf{X}_{Iw}$, the induced $S\times_{\Zp}\Fp$-point $\ol{x}$ lies in $Y_{Iw,1}^{0}$ (recall that this is a condition directly on the moduli functor). The same applies to $\mf{X}_{U(m),M}$. 

Let us now give a description of the correspondence attached to the Hecke operator $U_{p}$ at the level of formal schemes. We will do the construction for $\mf{X}_{U(m),M}$; the construction for $\mf{X}_{Iw,1}^{0}$ is almost word for word the same. Consider the contravariant functor $Nilp_{\Zp}^{ft} \ra Sets$ given by
\[
S \mapsto (A,\lambda, i, \ol{\eta}^{p}, \alpha, \mc{D})
\]
where $(A, \lambda, i, \ol{\eta}^{p}, \alpha)\in \mf{X}_{U(m),M}(S)$ and $\mc{D}\sub \mc{G}_{A}[p]$ is a finite flat subgroup scheme arising via $\alpha$ from a direct summand of $p^{-1}M/M \sub p^{-1}\epsilon\Lambda_{u}/\epsilon\Lambda_{u}$ in the way described by \cite[Lemma II.2.4(2)]{HT}. Note that the direct summand is \emph{not} part of the data, and that the finite type condition on $S$ ensures that we may apply \cite[Lemma II.2.4(2)]{HT}. We remark that such a $\mc{D}$ is \'etale of height $p^{n-1}$. This functor is represented by a formal scheme which we will denote by $\mf{Z}$. 

We have two natural finite flat maps $p_{1}, p_{2}\, :\, \mf{Z} \ra \mf{X}_{U(m),M}$: $p_{1}$ forgets $\mc{D}$ and $p_{2}$ quotients out $A$ by the subgroup scheme of $A[p]$ corresponding to $\mc{D}$ (for more details of how to take the quotient see \cite[p.109-110]{HT}). Note that $p_{1}$ is a bijection on $\ol{\mb{F}}_{p}$-points since $\mc{G}_{s}\cong \mu_{p^{\infty}} \times (\Qp/\Zp)^{n-1}$ for any $\ol{\mb{F}}_{p}$-point $s$.

\begin{lem}\label{trace}
The trace map ${\rm Tr}\, p_{1}\, :\, \mc{O}_{\mf{Z}} \ra \mc{O}_{\mf{X}_{U(m),M}}$ is divisible by $p^{n-1}$.
\end{lem}

\begin{proof}
The proof is an application of Serre-Tate theory and is due to Hida; we content ourselves with a brief sketch. To check the statement, it suffices to verify it at the level of completed local rings at $\ol{\mb{F}}_{p}$-points. Let $s$ be such a point. Then $\wh{\mc{O}}_{\mf{X}_{U(m),M},s}$ represents the deformation functor for $\mc{G}_{s}$ and using Serre-Tate theory (cf. \cite[Theorem 2.1]{Kat}) one finds that
\[
\wh{\mc{O}}_{\mf{X}_{U(m),M},s}(A) = \Hom_{\Zp}(T_{p}\mc{G}_{s} \otimes_{\Zp}T_{p}\mc{G}^{\vee}_{s}, 1+\mf{m}_{A})
\]
for any Artinian local ring $A$ with residue field $\ol{\mb{F}}_{p}$ and maximal ideal $\mf{m}_{A}$. Similarly, $\wh{\mc{O}}_{\mf{Z},s}$ represents the deformation functor for $\mc{G}_{s}$ plus the isogeny $\mc{G}_{s} \ra \mc{G}_{s}/\mc{D}$ where $\mc{D}$ is the unique \'etale subgroup scheme of $\mc{G}_{s}[p]$ of rank $p^{n-1}$. Serre-Tate theory shows that 
\[
\wh{\mc{O}}_{\mf{Z},s}(A) = \{ (q_{1},q_{2})\in \Hom_{\Zp}(T_{p}\mc{G}_{s} \otimes_{\Zp}T_{p}\mc{G}^{\vee}_{s}, 1+\mf{m}_{A})^{2} \mid q_{1}=q_{2}^{p} \}
\]
with $A$ as above. Choosing suitable coordinates this identifies $p_{1}\, :\, \wh{\mc{O}}_{\mf{X}_{U(m),M},s} \ra \wh{\mc{O}}_{\mf{Z},s}$ with the map $W(\ol{\mb{F}}_{p})[[T_{1},...,T_{n-1}]] \ra  W(\ol{\mb{F}}_{p})[[T_{1},...,T_{n-1}]]$ defined by $T_{i} \mapsto (1+T_{i})^{p}-1$ for all $i$, and one may show by direct computation that the trace of this map is divisible by $p^{n-1}$ (here $W(\ol{\mb{F}}_{p})$ denotes the Witt vectors of $\ol{\mb{F}}_{p}$).
\end{proof}

Next we define integral models for our spaces of $p$-adic automorphic forms. Once again we restrict the discussion to $\mf{X}_{U(m),M}$ but the details for $\mf{X}_{Iw}$ are the same. Let $\mf{A}_{U(m),M}$ denote the universal abelian variety over $\mf{X}_{U(m),M}$ and let $0$ denote its zero section. Let $\mf{L}=0^{\ast}{\rm Lie}(\mf{A}_{U(m),M}/\mf{X}_{U(m),M})$ and form the frame bundle
\[
\mf{T}=\underline{\mathrm{Isom}}(\mf{L},\mc{O}_{\mf{X}_{U(m),M}}\otimes_{\Zp}\mc{O}_{B,u}).
\]
Here both $\mf{L}$ and $\mc{O}_{\mf{X}_{U(m),M}}\otimes_{\Zp}\mc{O}_{B,u}$ carry $\mc{O}_{B,u}$-actions and pairings and we require the isomorphisms to respect these. This is an $L_{n-1,1}$-torsor on $\mf{X}_{U(m),M}$, where $L_{n-1,1}$ is the standard block diagonal Levi factor of the standard $(n-1,1)$ parabolic of $\GL_{n/\Zp}$ (which we think of as the elements of similitude factor $1$ inside the obvious integral model of $G_{\Qp}$). Any finite rank representation of $L_{n-1,1}$ over $\Zp$ gives us a vector bundle on $\mf{X}_{U(m),M}$ by the usual procedure. $L_{n-1,1}$ is isomorphic to $\GL_{n-1}\times \GL_{1}$ in an obvious way and thus its irreducible representations are parametrised by $n$-tuples $(k_{1},...,k_{n})\in \ZZ^{n}$ with $k_{1}\geq ... \geq k_{n-1}$. We will denote the vector bundle corresponding to such a tuple $(k_{1},...,k_{n})$ by $\mf{W}(k_{1},...,k_{n})$; by the setup one has 
\[
\mf{L}=\mf{W}(0,...,0,1)^{\oplus n} \oplus \mf{W}(0,...,0,-1,0)^{\oplus n}.
\]
Similarly one checks that 
\[
{\rm Lie}(A_{U(m)/\Qp}/X_{U(m)/\Qp})=W(0,...,0,1,1)^{\oplus n} \oplus W(0,...,0,-1,0,1)^{\oplus n}.
\]
From this we deduce that $\mf{M}(k_{1},...,k_{n})$ gives an integral structure to $W(k_{1},...,k_{n}, k_{n}-\sum_{i=1}^{n-1}k_{i})$ over the generic fibre of $\mf{X}_{U(m),M}$. To give an integral structure to arbitrary $W(k_{1},...,k_{n},w)$ one may observe that there is a natural isomorphism $\phi\, :\, W(k_{1},...,k_{n},w)\tilde{\longrightarrow} W(k_{1},...,k_{n},k_{n}-\sum_{i=1}^{n-1}k_{i})$ and use it to transport the integral structure. This is isomorphism is not Hecke-equivariant. For $U_{p}$ (which is the only operator we are interested in), one has the relation
\[
\phi\circ U_{p} = p^{-\frac{w+k_{1}+...k_{n-1}-k_{n}}{2}}U_{p} \circ \phi. 
\]
This follows since we are twisting by the similitude character to the power of $(w+k_{1}+...k_{n-1}-k_{n})/2$ to go from $W(k_{1},...,k_{n},k_{n}-\sum_{i=1}^{n-1}k_{i})$ to $W(k_{1},...,k_{n},w)$ (note that the exponent is an integer). We may now state purpose of this section:

\begin{prop}\label{divisibility}
The slopes of $U_{p}$ on $H^{0}(X_{U(m)}^{rig},W^{\dg}(k_{1},...,k_{n},w))$ and $H^{0}(X_{Iw}^{rig},W^{\dg}(k_{1},...,k_{n},w))$ are greater than or equal to
\[
-\frac{w+k_{1}+...k_{n-1}-k_{n}}{2}+n-1.
\]
\end{prop}

\begin{proof}
Again we content ourselves with a sketch. Note that by the equivariance relation above it suffices to show that the $U_{p}$-slopes are $\geq n-1$ when $w=k_{n}-\sum_{i=1}^{n-1}k_{i}$. Set $W=W(k_{1},..,k_{n},k_{n}-\sum_{i=1}^{n-1}k_{i})$ and $\mf{W}=\mf{W}(k_{1},...,k_{n})$ and define $W^{\dg}$ similarly. We embed $H^{0}(X_{U(m),M}^{rig},W^{\dg})$ into $H^{0}(X_{U(m)}^{ord},W)$ and show that any $U_{p}$-eigenvector in the latter space has slopes $\geq n-1$. To do this it suffices to prove that $U_{p}$ is divisible by $p^{n-1}$ on $H^{0}(\mf{X}_{U(m),M},\mf{W})$. The $U_{p}$-operator is defined as the composition
\[
H^{0}(\mf{X}_{U(m),M},\mf{W}) \overset{p_{2}^{\ast}}{\ra} H^{0}(\mf{Z},p^{\ast}_{2}\mf{W}) \ra H^{0}(\mf{Z},p^{\ast}_{1}\mf{W}) \overset{\mathrm{Tr}_{p_{1}}}{\ra} H^{0}(\mf{X}_{U(m),M},\mf{W})
\]
where the middle map comes from a natural isomorphism $p_{2}^{\ast}\mf{W}\cong p_{1}^{\ast}\mf{W}$ induced from the isomorphism of Lie algebras induced by the universal isogeny on $\mf{Z}$. The first two maps are integral (by definition) and the third is divisible by $p^{n-1}$ by Lemma \ref{trace}, so the proposition follows.
\end{proof}

\begin{rem}
Hida shows that the normalisation above is optimal. This is obvious in weight $(0,...,0,0)$ by looking constant functions. The general case follows from this by Hida's theory using the big space of $p$-adic automorphic forms (containing forms of all weights at once) and the existence of Hida families. We will not need this.
\end{rem}

\section{Computation of Cohomology and Classicality\label{sec:Computation-of-Cohomology}}

In this section we will compute the Euler characteristic $\sum_{i}(-1)^{i}H_{rig}^{i}(Y_{Iw,1}^{0},V^{\dg}(\xi)^{\vee})$
in two ways in the Grothendieck group of Hecke modules and deduce a control theorem for systems of Hecke eigenvalues of overconvergent automorphic forms of small slope. We will make use of rigid cohomology;
for a short recollection of the terminology we need see \cite[\S 4]{Joh} and for a good reference see \cite{LeSt}. In this section,
all integral models will be over $\Zp$, all characteristic
$0$ schemes and rigid analytic spaces will be over $\Qp$
and all characteristic $p$ schemes will be over $\Fp$
(so e.g. we will just write $X_{Iw}$ for what was previously called
$X_{Iw/\Qp}$).

\subsection{Computation in terms of overconvergent automorphic forms\label{sub: Comp-oc}}

Let us begin by recalling the generalized BGG resolution for the pair
$(\mf{g}=Lie(G),\mf{q}=Lie(Q))$. See \cite{Hum} for
a good reference on BGG resolutions for semisimple Lie algebras; the
extension to reductive Lie algebras is trivial and just amounts to
inserting a central character (that remains constant throughout the
resolution).
\begin{thm}
\label{thm:reductive BGG}(``generalized'' BGG resolution) If $\xi$
is an irreducible representation of the reductive Lie algebra $\mf{g}$
of dominant weight $\lambda=(k_{1},...,k_{n},w)$, then we have a
resolution
\[
0\to C_{n-1}^{\xi}\to...\to C_{0}^{\xi}\to\xi\to0
\]

with $C_{s}^{\xi}=U(\mf{g}) \otimes_{U(\mf{q})} M(w_{k}(\lambda+\rho)-\rho)$,
where $w_{s}$ for $0\leq s\leq n-1$ is the element of the Weyl group
$S_{n}$ of $G$ sending $(k_{1},...,k_{n},w)$ to $(k_{1},...,k_{n-s-1},k_{n-s+1},...k_{n},k_{n-s},w)$,
$\rho=\frac{1}{2}(n-1,n-3,...,1-n,0)$ is half the sum of the positive
roots for $G$, and $M(k_{1}^{\prime},...,k_{n}^{\prime},w^{\prime})$
for $k_{1}^{\prime}\geq...\geq k_{n-1}^{\prime}$, $w^{\prime} \equiv \sum_{i=1}^{n}k_{i}^{\prime}\,{\rm mod}\,2$
denotes the irreducible algebraic representation of $L$ with dominant
weight $(k_{1}^{\prime},...,k_{n}^{\prime},w^{\prime})$. The chain
complex $C_{\bullet}^{\xi}$ is a direct summand
of the bar resolution $D_{\bullet}^{\xi}$ of $\xi$ (where $D_{s}^{\xi}=U(\mf{g}) \otimes_{U(\mf{q})} \wedge^{s}(\mf{g/q}) \otimes_{\C}\xi$), and the inclusion is a quasi-isomorphism.
Moreover if $\xi$ is an irreducible representation of $G$, then
the above sequence is a resolution of $(U(\mf{g}),Q)$-modules
(and similarly for the quasi-isomorphism).
\end{thm}
Applying the automorphic vector bundle construction one obtains Faltings's
dual BGG complex for the vector bundle $V(\xi)$ with connection on
$X_{U}$ for any neat level $U$:
\begin{thm}
(\cite[Theorem 3]{Fal}, \cite{ChFa}, \cite{KP}) \label{thm: Faltings's BGG}If
$\xi$ is an irreducible representation of $G$ with dominant weight
$\lambda=(k_{1},...,k_{n},w)$ we have a complex
\[
0 \ra \mc{K}_{\lambda}^{0} \ra ... \ra \mc{K}_{\lambda}^{n-1} \ra 0
\]
called the dual BGG complex, with 
\[
\mc{K}_{\lambda}^{s}=W(w_{s}(\lambda+\rho)-\rho)^{\vee}=W(k_{1},...,k_{n-s-1},k_{n-s+1}-1,...,k_{n}-1,k_{n-s}+s,w)^{\vee},
\]
on $X_{U/\C}$ where the maps are Hecke-equivariant differential
operators. It is a direct summand of the de Rham
complex $V(\xi)^{\vee} \otimes_{\mc{O}_{X_{Iw/\C}}} \Omega_{X_{U/\C}}^{\bullet}$ and the inclusion is a quasi-isomorphism. Moreover, these constructions descend to any number field containing the reflex field  over which $\xi$ is defined. In particular, they may be base changed to $\Qp$ (because $G$ is split over $\Qp$ and $p$ is totally split in $F$).
\end{thm}
Next we apply $j_{Y_{Iw,1}^{0}}^{\dg}$ to the complex $\mc{K}_{\lambda}^{\bullet}$
to obtain a complex $\mc{K}_{\lambda}^{\dg,\bullet}=W^{\dg}(w_{s}(\lambda+\rho)-\rho)^{\vee}$
which, by the exactness of $j_{Y_{Iw,1}^{0}}^{\dg}$, is a quasi-isomorphic
of the overconvergent de Rham complex of $V^{\dg}(\xi)^{\vee}$
and hence may be used to compute the $H_{rig}^{i}(Y_{Iw,1}^{0},V^{\dg}(\xi)^{\vee})$. The same holds for higher levels using our frame $\ol{X}_{U(m),M}\sub \ol{X}_{U(m)} \sub \mf{X}_{U(m)}$.
\begin{prop}\label{compute} We have
\[
H_{rig}^{i}(Y_{Iw,1}^{0},V^{\dg}(\xi)^{\vee})=h^{i}(H^{0}(X_{Iw}^{rig},W^{\dg}(w_{\bullet}(\lambda+\rho)-\rho)^{\vee});
\]
\[ 
H_{rig}^{i}(\ol{X}_{U(m),M},V^{\dg}(\xi)^{\vee})=h^{i}(H^{0}(X_{U(m)}^{rig},W^{\dg}(w_{\bullet}(\lambda+\rho)-\rho)^{\vee});
\]
where $h^{i}$ stands for ``$i$-th cohomology of the complex''.\end{prop}
\begin{proof}
We do the Iwahori-level case; the proof in the higher level case is entirely analogous. By the discussion above we have
\[
H_{rig}^{i}(Y_{Iw,1}^{0},V^{\dg}(\xi)^{\vee})=H_{dR}^{i}(X_{Iw}^{rig},V^{\dg}(\xi)^{\vee})=H^{i}(X_{Iw}^{rig},\mc{K}_{\lambda}^{\dg,\bullet}).
\]
Next, we claim that the coherent cohomology
groups $H^{j}(X_{Iw}^{rig},W^{\dg}(w_{s}(\lambda+\rho)-\rho)^{\vee})$
vanish for $j\geq1$. This follows from the fact that $Y_{Iw,1}^{0}$
is affine as in \cite[Theorem 20]{Joh}. The affineness of $Y_{Iw,1}^{0}$
follows from that of $\ol{X}_{U,n-1}$ by finiteness of the
morphism $Y_{Iw,1}^{0} \ra \ol{X}_{U,n-1}$, and the affineness
of $\overline{X}_{U,n-1}$ follows from the fact that $\overline{X}_{U,n-1}$
is the non-vanishing locus of a non-zero section (the Hasse invariant)
of an ample line bundle on the projective variety $\overline{X}_{U}$
(see e.g. \cite[Proposition 7.8]{LS}).
\end{proof}

From now on let $\xi$ be the algebraic representation of $G$ with
dominant weight $\lambda=(1-n-k_{n},1-k_{n-1},...,1-k_{1},-w)$ for
$k_{1}$,...,$k_{n}$ such that $k_{1}\geq k_{2}\geq...\geq k_{n}+n$.
By the above we have $H_{rig}^{i}(Y_{Iw,1}^{0},V^{\dg}(\xi)^{\vee})=h^{i}(H^{0}(X_{Iw}^{rig},W^{\dg}(w_{\bullet}(\lambda+\rho)-\rho)^{\vee})$ and similarly in the higher level case.
Here $\xi^{\vee}$ has dominant weight $(k_{1}-1,...,k_{n-1}-1,k_{n}+n-1,w)$
and the dual BGG complex has, for $0\leq s\leq n-2$,
\[
W^{\dg}(w_{s}(\lambda+\rho)-\rho)^{\vee}=W^{\dg}(k_{1},...,k_{s},k_{s+2}-1,...,k_{n-1}-1,k_{n}+n-1,k_{s+1}-1-s,w)
\]
and $W^{\dg}(w_{n-1}(\lambda+\rho)-\rho)^{\vee}=W^{\dg}(k_{1},...,k_{n},w)$. 

Before we proceed let us make a few basic remark on slope decompositions. $U_{p}$ acts compactly on our spaces of overconvergent automorphic forms. On any vector space $V$ with a compact $U_{p}$-action there is, for any $h\in \R$, a canonical and functorial slope decomposition $V=V^{<h}\oplus V^{\geq h}$ where $V^{<h}$ is finite-dimensional and $U_{p}$ has slopes $<h$ on $V^{<h}$ and slopes $\geq h$ on $V^{\geq h}$. If $A$ is any commutative algebra of operators commuting with $U_{p}$, then the slope decomposition passes from the abelian category of finite-dimensional $A[U_{p}]$-modules to its Grothendieck group. We deduce the following consequence of Propositions \ref{divisibility} and \ref{compute}:

\begin{cor}
\label{cor: oc comp}With notation as above, set
\[
\alpha_{n-2}:=-\frac{w+k_{1}+...+k_{n-2}+k_{n}-k_{n-1}}{2}.
\]
Then 
\[
H^{i}_{rig}(Y_{Iw,1}^{0}, V^{\dg}(\xi)^{\vee})^{<\alpha_{n-2}}=\begin{cases} H^{0}(X_{Iw}^{rig},W^{\dg}(k_{1},...,k_{n},w))^{< \alpha_{n-2}} & \mbox{if } i=n-1 \\ 0  & \mbox{if }  i\neq n-1. \end{cases}
\]
Thus one has an equality 
\[
H^{0}(X_{Iw}^{rig},W^{\dg}(k_{1},...,k_{n},w))^{< \alpha_{n-2}}= \left( \sum_{i}(-1)^{n-1+i}H_{rig}^{i}(Y_{Iw,1}^{0},V^{\dg}(\xi)^{\vee}) \right) ^{< \alpha_{n-2}}
\]
as virtual $A[U_{p}]$-modules for any commutative algebra $A$ of Hecke operators commuting with $U_{p}$. The analogous statement holds in the higher level case. \end{cor}
\begin{proof}
We restrict ourselves to the Iwahori-level case for simplicity of notation; the proof is the same in both cases. As noted above, we have $H_{rig}^{i}(Y_{Iw,1}^{0},V^{\dg}(\xi)^{\vee}) = h^{i}(H^{0}(X_{Iw}^{rig},W^{\dg}(w_{\bullet}(\lambda+\rho)-\rho)^{\vee})$
and
\[
W^{\dg}(w_{s}(\lambda+\rho)-\rho)^{\vee} = W^{\dg}(k_{1},...,k_{s},k_{s+2}-1,...,k_{n-1}-1,k_{n}+n-1,k_{s+1}-1-s,w),
\]
for $0\leq s\leq n-2$ and $W^{\dg}(w_{n-1}(\lambda+\rho)-\rho)^{\vee}=W^{\dg}(k_{1},...,k_{n},w)$,
moreover the differentials in the dual BGG complex are Hecke-equivariant.
By Proposition \ref{divisibility}, the slopes of $U_{p}$ are greater
than or equal to
\[
\alpha_{s}:=-\frac{w+\left(\sum_{i\neq s+1}k_{i}\right)-k_{s+1}+2s+2}{2}+(n-1)
\]
on $H^{0}(X_{Iw}^{rig},W^{\dg}(w_{s}(\lambda+\rho)-\rho)^{\vee})$
for $0\leq s\leq n-2$ and greater than or equal to $\alpha_{n-1}:=-(w+k_{1}+...+k_{n-1}-k_{n})/2+n-1$
on $H^{0}(X_{Iw}^{rig},W^{\dg}(w_{n-1}(\lambda+\rho)-\rho)^{\vee})$.
We have that, for $0 \leq s \leq n-2$,
\[
\alpha_{s}-\alpha_{n-1}=k_{s+1}-(k_{n}+1+\frac{s}{2})
\]
and hence $\alpha_{1} \geq ... \geq \alpha_{n-2} \geq \alpha_{n-1}$. Since the formation of cohomology commutes with taking slope decompositions the first statement follows, and the second is a direct consequence of the first.\end{proof}
\begin{rem}
The ``usual" normalization of Hecke operators in the geometric theory of automorphic forms amounts, in our language, to choosing $w=k_{n}+2(n-1)-\sum_{i=1}^{n-1}k_{i}$. In this case $\alpha_{n-2}=k_{n-1}-k_{n}+1-n$.
\end{rem}

\subsection{Computation in terms of classical automorphic forms \label{sub:Comp-cl}}

We continue to assume that $\xi$ is the irreducible algebraic representation
of $G$ with dominant weight $\lambda=(1-n-k_{n},1-k_{n-1},...,1-k_{1},-w)$.
\begin{prop}
\label{prop:-Poincare D}$\sum_{i}(-1)^{d+i}H_{rig}^{i}(Y_{Iw,1}^{0},V^{\dg}(\xi)^{\vee})=\sum_{i}(-1)^{d+i}H_{rig,c}^{i}(Y_{Iw,1}^{0},V^{\dg}(\xi))^{\vee}$
as virtual $\mc{H}^{p}[\mc{U}^{-}]$-modules. \end{prop}
\begin{proof}
By Poincar\'e duality $H_{rig}^{i}(Y_{Iw,1}^{0},V^{\dg}(\xi)^{\vee})=H_{rig,c}^{2(n-1)-i}(Y_{Iw,1}^{0},V^{\dg}(\xi))^{\vee}$,
and Poincar\'e duality is Hecke equivariant since Hecke operators
act by correspondences. The proposition follows by taking the alternating sum.
\end{proof}
By definition
\[
\sum_{i}(-1)^{d+i}H_{rig,c}^{i}(Y_{Iw,1}^{0},V^{\dg}(\xi))^{\vee} = \left(\sum_{i}(-1)^{i}H_{rig,c}^{d+i}(Y_{Iw,1}^{0},V^{\dg}(\xi))\right)^{\vee}
\]
so we may calculate $\sum_{i}(-1)^{d+i}H_{rig}^{i}(Y_{Iw,1}^{0},V^{\dg}(\xi)^{\vee})$
via $\sum_{i}(-1)^{d+i}H_{rig,c}^{i}(Y_{Iw,1}^{0},V^{\dg}(\xi))$.
The calculation of the latter is essentially done by \cite[Theorem V.5.4]{HT}. To state and use it we will
need some notation, and we will also need to explain how to pass from the $\ell$-adic
setting of that theorem to the $p$-adic rigid setting that we are
in. Let us first begin by stating the special case we need:

\begin{prop}\label{V.5.4} For any $\xi$ we have an equality
\[
n \sum_{i}(-1)^{d+i} \varinjlim_{U^{p},m}H_{et,c}^{i}(Ig_{U^{p},m},V_{\ell}(\xi)) =  \sum_{i}(-1)^{d+i}\varinjlim_{U^{p},m}\iota_{\C,\ell}(H_{dR}^{i}(X_{U(m)/\C},V(\xi)))_{N_{M}^{op}}^{\Zp^{\times}} 
\]
of virtual $G_{M}(\mb{A}^{\infty})$-representations over $\ol{\qq}_{\ell}$ ($\ell\neq p$ a prime), where $V_{\ell}(\xi)$ is lisse $\ol{\qq}_{\ell}$-sheaf attached to $\xi$ and $\iota_{\ell}\, :\, \C \tilde{\ra} \ol{\qq}_{\C,\ell}$ is an isomorphism.
\end{prop}

\begin{proof}
This is a matter of specializing \cite[Theorem V.5.4]{HT} to our case; it is also a special case of  \cite[Theorem 6.7]{Shi4} (which has more transparent notation, the reader might prefer this). We freely use the notation of \cite[Theorem V.5.4]{HT}. In our case their $\rho$ is the trivial representation and $h=n-1$. One checks by definition that their function ${\rm Red}_{\rho}^{h}$ in this case reduces the unnormalized Jacquet functor $-_{N_{M}^{op}}$ (which is exact, so works well on virtual representations). One also checks by definition that the action of the group $G^{(h)}(\mb{A}^{\infty}$ used their is equivalent to our $G_{M}(\mb{A}^{\infty})$-action (here the $P_{M}$-action is trivial on $N_{M}$ on the right hand side).
\end{proof}

We now compare the left hand side in this proposition to $\sum_{i}(-1)^{d+i} \varinjlim_{m}H_{rig,c}^{i}(Ig_{U^{p},m},V^{\dg}(\xi))$ (still with a general $\xi$).

\begin{prop}\label{Deligne}
Fix $m$. Let $g\in G(\mb{A}^{p,\infty})\times (\ZZ \times \GL_{n-1}(\Qp))^{+} \times \Qp^{\times}$ and set $U_{Ig,m}=U^{p}\times \ZZ \times K_{n-1}(m) \Zp^{\times}$, where $K_{n-1}(m)\sub \GL_{n-1}(\Zp)$ is the subgroup of matrices reducing to the identity modulo $p^{m}$. Consider the double coset $U_{Ig,m}gU_{Ig,m}$ which acts as a cohomologcial correspondence on $H_{rig,c}^{i}(Ig_{U^{p},m},V^{\dg}(\xi))$ and $H_{et,c}^{i}(Ig_{U^{p},m},V_{\ell}(\xi))$. Then we have (writing ${\rm tr}$ for the trace)
\[
{\rm tr}(U_{Ig,m}gU_{Ig,m} \mid H_{rig,c}^{i}(Ig_{U^{p},m},V^{\dg}(\xi))) = {\rm tr}(U_{Ig,m}gU_{Ig,m} \mid H_{et,c}^{i}(Ig_{U^{p},m},V_{\ell}(\xi))) \in \qq .
\]
\end{prop}

\begin{proof}
Recall that $H_{et,c}^{i}(Ig_{U^{p},m},V_{\ell}(\xi))=a_{\xi}H^{i}_{et,c}(A^{m_{\xi}}_{Ig,m},\ol{\qq}_{\ell})(t_{\xi})$ and that this commutes with the actions, in the sense that the action of $U_{Ig,m}gU_{Ig,m}$ on the left hand side matches with the action of $c_{g,\xi}:=a_{\xi}U_{Ig,m}gU_{Ig,m}a_{\xi}$ on the right hand side. Here $A_{Ig,m}$ denotes the universal abelian variety over $Ig_{U^{p},m}$. We remark that $c_{g,\xi}$ is a linear combination of correspondences with rational coefficients (by construction); let us write $c_{g,\xi}=\sum_{j} \alpha_{j}c_{j}$ with $\alpha_{j}\in \qq$. Let $\varphi$ denote the Frobenius on $A_{Ig,m}$. Then, by to Fujiwara's trace formula (\cite{Fuj}) one has 
\[
{\rm tr}(\varphi^{N}c_{g,\xi}\mid H^{i}_{et,c}(A^{m_{\xi}}_{Ig,m}, \ol{\qq}_{\ell}))=\sum_{j} \alpha_{j}{\rm tr}(\varphi^{N}c_{j}\mid H^{i}_{et,c}(A^{m_{\xi}}_{Ig,m},\ol{\qq}_{\ell})) = \sum_{j}\alpha_{j}\# {\rm Fix}(\varphi^{N}c_{j}) \mid A_{Ig,m})
\]
for any $N \gg 0$, where ${\rm Fix}$ denotes the fixed point set. We note that the same argument applies to rigid cohomology, where instead of Fujiwara's trace formula we use its rigid analogue, due to Mieda (see \cite{Mie}; the details are to appear in work in preparation according to personal communication with Mieda). Thus we may conclude that
\[
{\rm tr}(\varphi^{N}c_{g,\xi}\mid H^{i}_{et,c}(A^{m_{\xi}}_{Ig,m}, \ol{\qq}_{\ell}))={\rm tr}(\varphi^{N}c_{g,\xi}\mid H^{i}_{rig,c}(A^{m_{\xi}}_{Ig,m}))\in \qq
\]
for large enough $N$ and it is standard to deduce that the same formula holds for all $N$, in particular $N=0$ (see e.g. the proof of \cite[Corollary 12.3.3]{Lau}). Twisting by $t_{\xi}$ we then obtain the result.
\end{proof}

\begin{cor}
For any $\xi$ we have an equality
\[
n \sum_{i}(-1)^{d+i} \varinjlim_{U^{p},m}H_{rig,c}^{i}(Ig_{U^{p},m},V^{\dg}(\xi)) =  \sum_{i}(-1)^{d+i}\varinjlim_{U^{p},m} \left( H_{dR}^{i}(X_{U(m)},V(\xi))^{\Zp^{\times}} \right) _{N_{M}^{op}} 
\]
of virtual $G_{M}(\mb{A}^{\infty})$-representations over $\Qp$.
\end{cor}

\begin{proof}
Since this may be checked by taking invariants under compact open subgroups and comparing traces of double coset operators, it follows directly from Propositions \ref{V.5.4} and \ref{Deligne}.
\end{proof}

\begin{cor}\label{cor2}
We have equalities (using our fixed $U^{p}$)
\[
n \sum_{i}(-1)^{d+i} H_{rig}^{i}(Y_{Iw,1}^{0},V^{\dg}(\xi)^{\vee}) =  \sum_{i}(-1)^{d+i}\left( \left( \varinjlim_{m}\left( H_{dR}^{i}(X_{U(m)},V(\xi))^{\Zp^{\times}}\right)_{N_{M}^{op}}\right)^{\vee} \right)^{\Zp^{\times}\times I_{n-1}}
\]
and
\[
n \sum_{i}(-1)^{d+i} H_{rig}^{i}(\ol{X}_{U(m),M},V^{\dg}(\xi)^{\vee}) =  \sum_{i}(-1)^{d+i}\left( \left( \varinjlim_{m}\left( H_{dR}^{i}(X_{U(m)},V(\xi))^{\Zp^{\times}}\right)_{N_{M}^{op}}\right)^{\vee} \right)^{\Zp^{\times}\times K_{n-1}(m)}
\]
(for all $m\geq 1$) of virtual $\mc{H}_{Ig}$-modules.
\end{cor}

\begin{proof}
For the first equality we take duals (using Proposition \ref{prop:-Poincare D} on the left hand side) and then take $\Zp^{\times} \times I_{n-1}$-invariants and identify $Ig_{Iw}$ with $Y_{Iw,1}^{0}$ Hecke-equivariantly. For the second we take duals as before and use Lemma \ref{tower} to replace the cohomology of the tower $(Ig_{U^{p},m})$ by that of $(\ol{X}_{U(m),M})$. We then take $\Zp^{\times}\times K_{n-1}(m)$-invariants.
\end{proof}

Our task is now to understand the right hand side better at $p$. To do this, we need to recall
the following general result:

\begin{lem}(\cite[Corollary 4.2.5]{Cas})
\label{lem: Jacquet n dual}Let $G$ be a connected reductive group
over $\Qp$, $P$ a parabolic with a Levi decomposition $P=LN$ and $\pi$ an irreducible
admissible representation of $G(\Qp)$. Let $P^{op}=LN^{op}$
be the opposite parabolic of $P$. Then $(\pi_{N^{op}})^{\vee} \cong (\pi^{\vee})_{N}$
as representations of $L$.
\end{lem}

We can now put the various technical results together to obtain:

\begin{prop}\label{prop: cl comp}
We have equalities
\[
n \sum_{i}(-1)^{d+i} H_{rig}^{i}(Y_{Iw,1}^{0},V^{\dg}(\xi)^{\vee}) \otimes \delta_{P_{M}}^{-1}=  \sum_{i}(-1)^{d+i} H_{dR}^{i}(X_{Iw},V(\xi)^{\vee})
\]
and
\[
n \sum_{i}(-1)^{d+i} H_{rig}^{i}(\ol{X}_{U(m),M},V^{\dg}(\xi)^{\vee})_{\mc{U}^{-}} \otimes \delta_{P_{M}}^{-1} =  \sum_{i}(-1)^{d+i} H_{dR}^{i}(X_{U^{\prime}(m)},V(\xi)^{\vee})_{\mc{U}^{-}}
\]
(for all $m\geq 1$) of virtual $\mc{H}^{p}[\mc{U}^{-}]$-modules.
\end{prop}
\begin{proof}
We start with the Iwahori case. The left hand side above is the left hand side of the formula in Corollary \ref{cor2}; we need to compute the right hand side. For simplicity, write $\pi(\xi)$ for $\sum_{i}(-1)^{d+i}\varinjlim_{m}H_{dR}^{i}(X_{U(m)},V(\xi))$ define $\pi(\xi^{\vee})$ similarly; the statement of the proposition boils down to showing that 
\[
(((\pi(\xi)^{\Zp^{\times}})_{N_{M}^{op}})^{\vee})^{\Zp^{\times}\times I_{n-1}}=\pi(\xi^{\vee})^{Iw_{p}}.
\]
By the preceding Lemma the left hand side is equal to $(((\pi(\xi)^{\Zp^{\times}})^{\vee})_{N_{M}})^{\Zp^{\times}\times I_{n-1}}$. Commuting the dual and the $\Zp^{\times}$-invariants we get $((\pi(\xi^{\vee})^{\Zp^{\times}})_{N_{M}})^{\Zp^{\times}\times I_{n-1}}$, and applying Lemma \ref{lem: BC} this is equal to $(\pi(\xi^{\vee})^{\Zp^{\times}})^{I_{n}}=\pi(\xi^{\vee})^{Iw_{p}}$, as desired. The higher level case is proved in exactly the same way, applying Lemma \ref{lem: BC2} instead of Lemma \ref{lem: BC}.
\end{proof}

From this theorem one may now deduce control theorems. The form of the result is the simplest and strongest when the highest weight of $\xi$ is regular. Numerically, this amounts to
$$ k_{1}>...>k_{n-1}>k_{n}+n. $$
We note that this assumption is rather harmless from the point of view of eigenvarieties. 

In what follows we put $\mf{g}_{\infty}=Lie(G(\R))$
and let $K_{\infty}$ be a maximal compact modulo center subgroup
of $G(\R)$. For any $(\mf{g}_{\infty},K_{\infty})$-module
$\tau$, we let $H^{i}(\mf{g}_{\infty},K_{\infty},\tau)$ denote
the $i$-th $(\mf{g}_{\infty},K_{\infty})$-cohomology of $\tau$
(see \cite{BW}). We may now state and prove the main result of this paper:

\begin{thm}
\label{thm: main}
Assume that $k_{1}>...>k_{n-1}>k_{n}+n$. Then 
$$ H^{0}(X_{Iw}^{rig}, W(k_{1},...,k_{n},w))^{< \alpha_{n-2}} = H^{0}(X_{Iw}^{rig}, W^{\dg}(k_{1},...,k_{n},w))^{< \alpha_{n-2}} $$
and
$$ H^{0}(X_{U(m)}^{rig}, W(k_{1},...,k_{n},w))_{\mc{U}^{-}}^{< \alpha_{n-2}} = H^{0}(X_{Iw}^{rig}, W^{\dg}(k_{1},...,k_{n},w))_{\mc{U}^{-}}^{< \alpha_{n-2}}, $$
where we recall that 
$$ \alpha_{n-2} = -\frac{w+k_{1}+...+k_{n-2}+k_{n}-k_{n-1}}{2}. $$
In other words, any generalised overconvergent $U_{p}$-eigenform of weight $(k_{1},...,k_{n},w)$ and slope less than $\alpha_{n-2}$ (on which the whole $\mc{U}^{-}$ acts invertibly in the higher level case) is classical.
\end{thm}

\begin{proof}
Let us do the higher level case, the proof in the Iwahori case being almost identical (and slightly simpler).  By Corollary \ref{cor: oc comp} we have an equality
$$ n.H^{0}(X_{U(m)}^{rig}, W^{\dg}(k_{1},...,k_{n},w))_{\mc{U}^{-}}^{< \alpha_{n-2}} = \left( n \sum_{i} (-1)^{n-1+i} H_{rig}^{i}(\ol{X}_{U(m),M}, V^{\dg}(\xi)^{\vee})_{\mc{U}^{-}}\otimes \delta_{P}^{-1} \right)^{<\alpha_{n-2}} $$
of virtual $\mc{H}^{p}[\mc{U}^{-}]$-modules. Applying Proposition \ref{prop: cl comp} to the right hand side we obtain the equality
$$ n.H^{0}(X_{U(m)}^{rig}, W^{\dg}(k_{1},...,k_{n},w))_{\mc{U}^{-}}^{< \alpha_{n-2}} = \left( \sum_{i} (-1)^{n-1+i} H_{dR}^{i}(X_{U^{\prime}(m),M}, V(\xi)^{\vee})_{\mc{U}^{-}} \right)^{<\alpha_{n-2}} $$
of virtual $\mc{H}^{p}[\mc{U}^{-}]$-modules. By regularity of the highest weight $H_{dR}^{i}(X_{U^{\prime}(m),M}, V(\xi)^{\vee})=0$ unless $i=d$ (by work of Vogan-Zuckerman), so the formula simplifies to 
$$ n.H^{0}(X_{U(m)}^{rig}, W^{\dg}(k_{1},...,k_{n},w))_{\mc{U}^{-}}^{< \alpha_{n-2}} = H_{dR}^{d}(X_{U^{\prime}(m),M}, V(\xi)^{\vee})_{\mc{U}^{-}}^{<\alpha_{n-2}}. $$
By Matsushima's formula 
$$ H_{dR}^{d}(X_{U^{\prime}(m)}(\mb{C}), V(\xi)^{\vee})=\bigoplus_{\pi}m(\pi)(\pi^{\infty})^{U^{\prime}(m)} \otimes H^{d}(\mf{g}_{\infty},K_{\infty}, \pi_{\infty} \otimes \xi^{\vee}). $$
where $\pi=\pi^{\infty}\otimes \pi_{\infty}$ (finite resp. infinite part) runs over the irreducible admissible representations of $G(\mb{A})$ and $m(\pi)$ denotes the multiplicity of $L^{2}$-space of automorphic forms. Results of Vogan-Zuckerman and Kottwitz (see the discussion after Theorem 1 of \cite{Kot}) imply that the $H^{d}(\mf{g}_{\infty},K_{\infty}, \pi_{\infty} \otimes \xi^{\vee})=0$ unless $\pi_{\infty}$ belongs to the discrete series $L$-packet associated with $\xi$, and that, for fixed $\pi^{\infty}$, the multiplicity $m(\pi^{\infty}\otimes \pi_{\infty})$ is constant as $\pi_{\infty}$ varies in that $L$-packet. Writing $m(\pi^{\infty})$ for this common multiplicity we may rewrite the above as
$$ H_{dR}^{d}(X_{U^{\prime}(m)}(\mb{C}), V(\xi)^{\vee})=\bigoplus_{\pi^{\infty}} m(\pi^{\infty})(\pi^{\infty})^{U^{\prime}(m)} \otimes \bigoplus_{\pi_{\infty}} H^{d}(\mf{g}_{\infty},K_{\infty}, \pi_{\infty} \otimes \xi^{\vee}). $$
Once again by results of Vogan-Zuckerman, the space $\bigoplus_{\pi_{\infty}} H^{d}(\mf{g}_{\infty},K_{\infty}, \pi_{\infty} \otimes \xi^{\vee})$ has dimension $n$ and it's $(p,q)$-decomposition is regular (i.e. all non-zero $(p,q)$-spaces have dimension $1$); see e.g. \cite[Corollary VI.6.27(3)]{HT} for similar results. Since the holomorphic part corresponds to the contribution from holomorphic automorphic forms we may conclude that 
$$ H_{dR}^{d}(X_{U^{\prime}(m)}(\mb{C}), V(\xi)^{\vee})=n.H^{0}(X_{U^{\prime}(m)}(\mb{C}),W(k_{1},...,k_{n},w)) $$
as $\mc{H}^{p}(\mc{U}^{-})$-modules. Together with the previous discussion and the GAGA-principle we may conclude that
$$ H^{0}(X_{U^{\prime}(m)}^{rig},W(k_{1},...,k_{n},w))_{\mc{U}^{-}}^{< \alpha_{n-2}} = H^{0}(X_{U^{\prime}(m)}^{rig}, W^{\dg}(k_{1},...,k_{n},w))_{\mc{U}^{-}}^{< \alpha_{n-2}} $$
as virtual $\mc{H}^{p}(\mc{U}^{-})$-modules, and hence as honest $\mc{H}^{p}(\mc{U}^{-})$-modules (e.g. by counting dimensions, since the left hand side is a submodule of the right hand side). This finishes the proof.
\end{proof}

\begin{rem} We end with a few remarks:

\begin{enumerate}

\item The condition of $\mc{U}^{-}$-invertibility in the higher level case may be dropped (by adjusting Lemma \ref{lem: BC2}) but, as far as the author is aware of, this condition is always satisfied in the context of eigenvarieties.

\item When $\xi$ has non-regular highest weight one may still prove that any finite slope overconvergent eigenform has system of Hecke eigenvalues occurring in the space of automorphic forms.

\item Knowing the control theorem for generalized eigenforms, as opposed to just eigenforms, is important in some applications (e.g. in Chenevier's strategy to prove that the weight map on an eigenvariety is \'etale at non-critical points).  

\end{enumerate}

\end{rem}

\end{document}